\documentclass[10pt]{amsart}

\usepackage{amssymb}
\usepackage[all]{xy}
\usepackage[colorlinks=true, urlcolor=rltblue] {hyperref}
\usepackage{color}
\definecolor{rltblue}{rgb}{0,0,0.75}
\usepackage{pdfsync}

\newtheorem{thm}{Theorem}[section]
\newtheorem{theorem}[thm]{Theorem}
\newtheorem{lemma}[thm]{Lemma}
\newtheorem{proposition}[thm]{Proposition}

\theoremstyle{definition}

\newtheorem{definition}[thm]{Definition}

\theoremstyle{remark}
\newtheorem{obs}[thm]{Observation}

\newtheorem{claim}[thm]{Claim}

\theoremstyle{plain}

\def\om{\omega}

\def\S{\mathcal S}
\def\F{\mathcal F}
\def\a{\alpha}
\def\b{\beta}
\def\g{\gamma}

\def\si{\sigma}

\def\Nat{{\mathbb N}}
\def\isom{\cong}
\def\O{{\mathcal O}}

\def\Si{\Sigma}

\def\la{\langle}
\def\ra{\rangle}

\def\om{\omega}

\def\a{\alpha}
\def\b{\beta}
\def\g{\gamma}

\def\si{\sigma}

\def\Nat{{\mathbb N}}
\def\isom{\cong}
\def\O{{\mathcal O}}

\def\Si{\Sigma}
\def\ra{\rangle}
\def\la{\langle}

\newcommand \omck{\omega_1^{\textup{CK}}}
\def\Os{\O^*}

\newcommand \seq[1]{{\left\langle{#1}\right\rangle}}
\newcommand{\floor}[1]{{\left\lfloor #1 \right\rfloor}} 
\newcommand{\rest}[1]{\! \upharpoonright\!{#1}} 
\newcommand \tth{{}^{\textup{th}}}

\newcommand{\cero}{\mathbf{0}}

\newcommand{\w}{\omega}
\newcommand \s{\sigma}
\renewcommand \le {\leqslant}
\renewcommand \leq {\leqslant}
\renewcommand \ge {\geqslant}
\renewcommand \geq {\geqslant}

\newcommand \andd{\,\,\,\&\,\,\,}

\newcommand \Rat{\mathbb{Q}}

\newcommand \degd{\mathbf d}
\newcommand \degb{\mathbf b}
\newcommand \dega{\mathbf a}
\newcommand \MM{\mathcal M}
\newcommand \CC{\mathcal C}
\newcommand \DD{\mathcal D}
\newcommand \GG{\mathcal G}
\newcommand \OO{\mathcal O}
\newcommand \NN{\mathcal N}
\newcommand \KK{\mathcal K}
\newcommand \AAA{\mathcal A}
\newcommand \LL{\mathcal L}
\newcommand \FF{\mathcal F}
\newcommand \nle{\nleqslant}
\DeclareMathOperator \low{low}
\newcommand \ZFC{\textup{ZFC}}
\newcommand \Int{\mathbb Z}
\newcommand \Tur{\textup{T}}
\newcommand \SSS{\mathcal S}
\newcommand \PP{\mathcal P}

\newcommand \rooot{\texttt{root}}
\DeclareMathOperator \range{range}

\DeclareMathOperator\supr{\mathtt{sup}}
\DeclareMathOperator\rk{\mathtt{rk}}
\DeclareMathOperator\fat{\mathtt{fat}}
\DeclareMathOperator\mini{\mathtt{min}}
\DeclareMathOperator\hht{\mathtt{ht}}

\DeclareMathOperator \Spec{Spec}

\newcommand \integer[1]{\floor{#1}}

\title{Relative to any non-hyperarithmetic set}

\author{Noam Greenberg}
\address{School of Mathematics, Statistics and Operations Research \\ Victoria University of Wellington \\
New Zealand}
\email{greenberg@msor.vuw.ac.nz}

\author{Antonio Montalb\'an}

\address{Department of Mathematics\\
University of Chicago\\
5734 S. University ave.\\
Chicago, IL 60637, USA}

\email{antonio@math.uchicago.edu}

\author{Theodore A. Slaman}
\address{Department of Mathematics, University of California Berkeley, Berkeley, Califor- 
nia 94720-3840 }
\email{slaman@math.berkeley.edu}

\thanks{The first author was partially supported by a Marsden grant from the royal society of New Zealand.
The second author was partially supported by NSF grant DMS-0901169 and the Packard Fellowship.}

\begin{document}

\maketitle

\today

\begin{abstract}
We prove that there is a structure, indeed a linear ordering, whose degree spectrum is the set of all non-hyperarithmetic degrees. We also show that degree spectra can distinguish measure from category. 
\end{abstract}

%

\section{Introduction}

Slaman \cite{Sla98} and Wehner \cite{Weh98} independently proved the following theorem (with a third proof later provided by Hirschfeldt \cite{Hir06}).

\begin{theorem}\label{thm_Slaman_Wehner}
	There is a countable structure $\MM$ such that for any set $X$, $X$ computes an isomorphic copy of $\MM$ if and only if $X$ is not computable. 
\end{theorem}

Theorem \ref{thm_Slaman_Wehner} fits into the general research programme of classifying \emph{degree spectra} of countable structures. Recall that the (Turing) degree of a structure $\MM$ whose universe is a subset of the set of natural numbers $\w$ is defined to be the degree of its atomic (equivalently, quantifier-free) diagram; and that the \emph{degree spectrum} $\Spec(\MM)$ of a countable structure is defined to be the collection of Turing degrees which compute an isomorphic copy of $\MM$. Classifying degree spectra of countable structures amounts to finding which computability-theoretic aspects of a set $X$ of natural numbers are reflected in the isomorphism types of countable structures which $X$ can effectively present. Equivalently, the question is how much of the richness of the structure of the power set of the reals (or the Turing degrees) is reflected in the substructure consisting only of the degree spectra of countable structures. In other words, we ask which properties of sets of Turing degrees can be discerned by restricting to nicely definable sets; in this instance, ``nicely definable'' means being a degree spectrum of a structure, which is a particularly nice $\mathbf{\Sigma}^1_1$ definition of a set. In the example above, the Slaman-Wehner theorem says that being noncomputable is a property which is so definable. Even though there is no noncomputable set of least Turing degree, there is a single countable structure whose isomorphism type captures what it means for a set to contain any noncomputable information. 

Among properties of reals, being noncomputable is a ``large'' property, since the collection of noncomputable reals is co-countable. Recent efforts were devoted to understanding large degree spectra. Among the co-countable classes, Kalimullin \cite{Kalimullin-low,Kalimullin-ce,Kalimullin-restrict} investigated complements of lower cones: relativisations of the Slaman-Wehner theorem to nonzero degrees. He showed that such a relativisation holds to any low and any c.e.\ degree, but not to some $\Delta^0_3$ degree. Goncharov, Harizanov, Knight, McCoy, Miller and Solomon \cite{GHKMMS} showed that for any computable ordinal $\alpha$, the collection of non-$\low_\alpha$ degrees -- those degrees $\dega$ for which $\dega^{(\alpha)}>\cero^{(\alpha)}$ -- is a degree spectrum. Among classes which are not co-countable, we can appeal to the notions of category and measure to obtain notions of largeness, namely being co-meagre or being co-null. For example, Csima and Kalimullin \cite{CsimaKalimullin} showed that the collection of hyperimmune degrees is a degree spectrum. This is a collection which is both co-null and co-meagre, but is not co-countable. 

The largeness notions given by category and measure are not compatible: there is a meagre co-null class, and also a null co-meagre class. In Section \ref{sec_separating} we show that this incompatibility is reflected in degree spectra; namely that there is a null and co-meagre degree spectrum (Theorem \ref{thm_null_co_meagre}), and a meagre and co-null spectrum (Theorem \ref{thm_meagre_co_null}). 

It is not difficult to see that there can be only countably many co-null or co-meagre degree spectra. In fact, Kalimullin and Nies, and independently the authors \cite{GMS} showed that any co-null degree spectrum must include the Turing degree of Kleene's $\OO$, the complete $\Pi^1_1$ set. One then wonders if this bound can be improved to be hyperarithmetic. Our main theorem shows, in a strong way, that it cannot. 

\begin{theorem} \label{main_theorem}
There is a structure whose degree spectrum is the set of all non-hyperarithmetic degrees.
\end{theorem}

This result has implications for higher degree structures. In \cite{GMS} it is shown that the Slaman-Wehner theorem fails for the degrees of constructibility: under standard mild set-theoretic assumptions, if for any non-constructible set $X\subseteq \w$, a countable structure $\MM$ has a copy constructible in $X$, then $\MM$ has a constructible copy. Theorem \ref{main_theorem}, though, shows that the analogue of the Slaman-Whener theorem does hold in the hyperdegrees. In other words, there is a structure which is $\Delta^1_1(X)$-presentable in every nonhyperarithmetic set $X$, but has no hyperarithmetic copy; so the hyperspectrum of this structure consists of the nonzero hyperdegrees. 

Two lines of continued research present themselves. Further restrictions on definability arise from looking at degree spectra of structures in particular classes. It is unknown, for example, if the Slaman-Wehner theorem can be witnessed by a linear ordering. We show (Theorem \ref{thm_LO}) that Theorem \ref{main_theorem} can: there is a countable linear ordering whose degree spectrum is the collection of non-hyperarithmetic degrees. On the other hand, Theorem \ref{main_theorem} cannot be witnessed by a countable model of an uncountably categorical structure. The argument is simple: if the degree spectrum of $\MM$ contains all non-hyperarithmetic degrees, then the theory of $\MM$ is hyperarithmetic; because the theory of $\MM$ is hyperarithmetic in any copy of $\MM$, and there is a minimal pair of hyperdegrees. Khisamiev \cite{Khisamiev} and Harrington \cite{Harrington_prime}, however, showed that if $T$ is a countable, uncountable categorical theory, then every countable model of $T$ has a copy computable in $T$. We wonder if the countable structures which witness Theorem \ref{main_theorem} fall on one side of some standard watershed of the stability spectrum. 

We also ask which other complements of countable ideals of Turing degrees are degree spectra. There are only countably many such ideals. Is the ideal of arithmetic degrees a complement of a degree spectrum? Kalimullin's original problem -- for which degrees $\dega$ is there a structure $\MM$ whose degree spectrum is the collection of degrees $\degb \nle \dega$ -- remains open. There are only countably many such degrees, but we do not know any upper bound on them. Are they all hyperarithmetic? The next natural test case is finding whether $\cero''$ is such a degree. A hazardous guess, based on a proof that all c.e.\ degrees are such degrees, asks if these degrees coincide with the Turing degrees which contain ranked sets.

\section{Relative to any non-hyperarithmetic degree}

In this section we prove Theorem \ref{main_theorem}. 

\subsection{Discussion}

Goncharov, Harizanov, Knight, MacCoy, R. Miller and Solomon \cite{GHKMMS} showed that for any computable ordinal $\alpha$, the collection of non-$\low_\alpha$-degrees, namely those degrees $\degd$ such that $\degd^{(\alpha)}>\cero^{(\alpha)}$, is the degree spectrum of a structure~$\MM_\alpha$. We observe that a degree $\degd$ is hyperarithmetic if and only if it is $\low_\alpha$ for some computable ordinal $\alpha$: certainly every $\low_\alpha$-degree is computable in $\cero^{(\alpha)}$, and so is hyperarithmetic; and for all computable $\alpha$, $\cero^{(\alpha)}$ is $\low_{\alpha\cdot \w}$, as 
\[ \left(\cero^{(\alpha)}\right)^{(\alpha\cdot \w)} = \cero^{(\alpha+\alpha\cdot \w)} = \cero^{(\alpha\cdot \w)}.\]
It follows that a degree $\degd$ is non-hyperarithmetic if and only if it computes a copy of $\MM_\alpha$ for every computable $\alpha$, and in fact, an examination of the proof shows that such a copy is computable uniformly from $\degd$ and $\alpha$. In other words, there is a Turing functional $\Phi$ such that for any non-hyperarithmetic arithmetic set $D$ and every notation $a$ for an ordinal $\alpha$ in some $\Pi^1_1$ path in Kleene's $\OO$, $\Phi(D,a)$ is a copy of $\MM_\alpha$. 

It would then seem natural to try to prove Theorem \ref{main_theorem} by examining the disjoint union $\MM$ of all the structures $\MM_\alpha$. Certainly any degree computing a copy of $\MM$ must be non-hyperarithmetic, and a non-hyperarithmetic degree can compute every component of $\MM$. However, a non-hyperarithmetic degree which does not compute Kleene's $\OO$ will not know to apply $\Phi$ only to notations of computable ordinals. It will therefore fail to compute a copy of $\MM$. 

To overcome this problem, we examine what happens when $\Phi$ is applied to notations for nonstandard ordinals. Restricting to a computable set of such notations, which is obtained from an overspill argument, overcomes the problem of working with the $\Pi^1_1$-set of notations. Essentially, we show that the structure $\Phi(D,a)$ for notations $a$ for a nonstandard ordinal is in fact computable, and does not depend on $D$ or on $a$. In some sense it is the ``ill-founded limit'' of the structures $\MM_\alpha$. Adding this limit as a component to $\MM$ results in the structure we are seeking.

\subsection{The proof}

The proof of Theorem \ref{main_theorem} relies on three ingredients: nonstandard ordinals, the Wehner graph, and jump inversion for graphs.

\subsubsection{Nonstandard ordinals}

H.~Friedman (see \cite[III,1.9]{Sacks}) used the Gandy hyperlow basis theorem to construct a countable $\w$-model of set theory which omits the least non-computable ordinal $\omck$. We fix such a model $H$. So $H$ is a model of $\ZFC$ (by which we mean, of course, that $H$ is a model of a finite fragment of $\ZFC+V=L$, sufficiently strong for our purposes). The well-founded part of $H$ extends $L_{\omck}$ but has height $\omck$. There are non-well-ordered computable orderings of $\w$ which have no infinite descending sequences in $H$. These are the order-types of non-standard computable ordinals of $H$. 

For all $Z\in 2^\w\cap H$, and every $H$-notation $\alpha\in \OO^H$, the $\alpha\tth$ iteration $Z^{(\alpha)}$ of the Turing jump of $Z$ is a well-defined element of $2^\w\cap H$. This is compatible with the standard definition: for a standard notation $\alpha$, $Z^{(\alpha)}$ as computed in $H$ agrees with the value computed in $V$. This follows by induction on $\alpha$, or simply by absoluteness of $\mathbf{\Delta}^1_1(\ZFC)$ facts between $H$ and $V$. 

We fix an ill-founded computable ordinal $\delta^*$ of $H$. We identify $\delta^*$ with an element of $\OO^H$, the collection of $a\in \Nat$ which $H$ believes are notations for computable ordinals. From this notation, which we also call $\delta^*$, we can obtain a computable linear ordering of $\w$ which is isomorphic to $\delta^*$; we call this linear ordering $\delta^*$ as well. An ordering such as $\delta^*$, which is ill-founded but has no infinite descending hyperarithmetic sequences, was first constructed by Harrison~\cite{Har68}, who also showed that his ordering supports a jump hierarchy. The maximal well-founded initial segment of~$\delta^*$ is $\omck$. We choose $\delta^*$ so that in $H$, it is a limit ordinal which is closed under ordinal addition and multiplication; and we choose the linear ordering $\delta^*$ so that the operations of ordinal addition and multiplication on the left by $\w$ are computable on $\delta^*$. This can be achieved, for example, by taking $\delta^*$ to be a power of a power of $\w$. 


\subsubsection{The relativised Wehner graph}

A relativisation of Wehner's proof of the Slaman-Wehner theorem (\ref{thm_Slaman_Wehner}) yields, for any set $X$ of natural numbers, a (simple undirected) graph $G_X$ whose degree spectrum does not contain $\deg_\Tur(X)$ but does contain the degrees strictly above $\deg_\Tur(X)$. The construction yields a \emph{total} operator: a Turing functional $\Phi$ such that for all oracles $X$, $\Phi(X)\in 2^\w$ is total. 

\begin{theorem}\label{thm_relativised_Wehner}
	There is a total Turing functional $\Phi$ such that for all $X$ and $Y$ in $2^\w$, $\Phi(Y,X)$ is a graph, and $\Phi(Y,X)\cong G_X$ if and only if $Y \nle_\Tur X$. 
\end{theorem}

\subsubsection{Jump inversion for graphs} \label{subsubsec_jump_inversion}

The structures $\MM_\alpha$ mentioned above were obtained in \cite{GHKMMS} by inverting the Wehner graph $G_{\emptyset^{(\alpha)}}$. Replacing the edges of the graph by pairs of linear orderings built from $\Int^\alpha$, the authors show an $\alpha$-jump inversion for graphs: an operation taking a graph $G$ to a structure $\NN$ such that for all $X\in 2^\w$, $X$ computes a copy of $\NN$ if and only if $X^{(\alpha)}$ computes a copy of $G$. 

We show that this jump inversion can be stretched to the nonstandard ordinals, in which case it produces a uniform computable structure. To state the following theorem in a uniform fashion for both finite and infinite computable ordinals, we need to shift the infinite indices by 1 (as is done in \cite{AK00}). We discuss this in greater length in Subsection \ref{subsubsec_corrected}. Let $Z\in 2^\w$. For $n<\w$, we let $Z_{(n)} = Z^{(n)}$. For $\alpha\in [\w,\omck)$, we let $Z_{(\alpha)} = Z^{(\alpha+1)}$. If $Z\in H$, then for all $\alpha\in (\omck)^H\setminus \omck$, we also let $Z_{(\alpha)}= Z^{(\alpha+1)}$. 

\begin{theorem}\label{thm_jump_inversion}
	For any graph $G$ and $\alpha<\delta^*$ there is a structure $G^{-\alpha}$ (in a fixed language $\LL$) with the following properties:
	\begin{enumerate}
		\item For $\alpha<\omck$, for any $X\in 2^\w$, if $X$ computes a copy of $G^{-\alpha}$, then $X_{(2\alpha)}$ computes a copy of $G$.\footnote{In fact, a set $X\in 2^\w$ computes a copy of $G^{-\alpha}$ if and only if $X_{(2\alpha)}$ computes a copy of $G$. This follows from a relativisation of a lemma below; but this equivalence is not necessary for the proof of Theorem \ref{main_theorem}.}
		\item For $\alpha,\beta \in \delta^*\setminus \omck$, and any graphs $G$ and $H$, $G^{-\alpha}\cong H^{-\beta}$.
	\end{enumerate}
	Moreover, for any total Turing functional $\Phi$ there is a Turing functional $\Psi$ such that for all $X\in 2^\w$, $\alpha< \delta^*$ and all graphs $G$, if $\Phi(X,\emptyset_{(2\alpha)})\cong G$ then $\Psi(X,\alpha)\cong G^{-\alpha}$. 
\end{theorem}

The isomorphism type of $G^{-\alpha}$ depends only on $\alpha$ and the isomorphism type of~$G$. We fix a structure $G^{-\infty}$ such that for any graph $H$ and any $\alpha\in \delta^*\setminus \omck$, $G^{-\infty}\cong H^{-\alpha}$. We will see that $G^{-\infty}$ has a computable copy.

\subsubsection{The proof of Theorem \ref{main_theorem}}

Before we present the proofs of Theorems \ref{thm_relativised_Wehner} and \ref{thm_jump_inversion}, we show 
how to use them to prove Theorem \ref{main_theorem}. The structure $\AAA$ whose degree spectrum is the 
collection of non-hyperarithmetic degrees is the disjoint union of $G_{\emptyset_{(2\alpha)}}^{-\alpha}$ for 
$\alpha<\delta^*$. In other words, $\AAA$ is the disjoint union of $G_{\emptyset_{(2\alpha)}}^{-\alpha}$ for 
$\alpha<\omck$ and infinitely many copies of $G^{-\infty}$. Formally, the universe of $\AAA$ is the disjoint 
union of the universes of the structures $\NN_\alpha = G_{\emptyset_{(2\alpha)}}^{-\alpha}$ for $\alpha<\delta^*$, 
with an added equivalence relation whose equivalence classes are the various $\NN_\alpha$'s.

We show that a set $X$ computes a copy of $\AAA$ if and only if it is not hyperarithmetic. 

For $\alpha<\omck$, if a set $X$ computes a copy of $\NN_\alpha$ then $X_{(2\alpha)}$ computes a copy of $G_{\emptyset_{(2\alpha)}}$, whence $X_{(2\alpha)} >_\Tur \emptyset_{(2\alpha)}$. As we argued above, if~$X$ is hyperarithmetic, then there is some computable ordinal $\beta$ such that for all computable ordinals $\gamma\ge \beta$, $X_{(\gamma)} \equiv_\Tur \emptyset_{(\gamma)}$. It follows that if $X$ is hyperarithmetic, then there are computable ordinals $\gamma$ such that $X$ does not compute a copy of $\NN_\gamma$, and so $X$ does not compute~$\AAA$. 

On the other hand, if $X$ is not hyperarithmetic, then for all $\alpha<\omck$, $X\nle_\Tur \emptyset_{(2\alpha)}$, and so $\Phi(X,\emptyset_{(2\alpha)})\cong G_{\emptyset_{(2\alpha)}}$, where $\Phi$ is the functional guaranteed by Theorem \ref{thm_relativised_Wehner}. If $\Psi$ is the functional obtained from $\Phi$ by Theorem \ref{thm_jump_inversion}, then it follows that $\Psi(X,\alpha)\cong \NN_\alpha$. Certainly for $\alpha\in \delta^*\setminus \omck$, we have $\Psi(X,\alpha)\cong G^{-\infty}$. Hence the disjoint union of the structures $\Psi(X,\alpha)$ for $\alpha<\delta^*$ is an $X$-computable structure isomorphic to $\AAA$.

\medskip

This completes the proof of Theorem \ref{main_theorem}. We can think of this proof as follows. The component $\NN_\alpha$ of $\AAA$ can be thought of as the result of diagonalising $\AAA$ against the $\alpha\tth$ hyperarithmetic structure. As we shall shortly see, the graph $G_{\emptyset_{(2\alpha)}}$ is obtained by aggregating all possible ways that a finite initial segment of a set $X\nle_\Tur \emptyset_{(2\alpha)}$ codes into the graph the fact that it enumerates a set that $\emptyset_{(2\alpha)}$ cannot enumerate. This process ensures that the result does not depend on $X$. The jump-inversion operation then translates this information, coded directly into the atomic diagram of $G_{\emptyset_{(2\alpha)}}$, into the $\Sigma^0_{2\alpha}$-diagram of $\NN_\alpha$. %
This translation involves a certain amount of ``overwriting''; the raw information present in the graph $G_{\emptyset_{(2\alpha)}}$ is homogenised, or diluted, so that only an iteration of the jump of length $2\alpha$, and no shorter, can recover it. If $\alpha$ is nonstandard, then this dilution process completely overwrites all the information in the graph produced by $X$ as it attempts to diagonalise against $\emptyset_{(2\alpha)}$. We have no control over this graph, as it is possible that $X$ is nonhyperarithmetic yet is computable from $\emptyset_{(2\alpha)}$. Overwriting all the information coded in this graph ensures that we produce $\NN_\alpha \cong G^{-\infty}$, which again doesn't depend on $X$. 

Note that in our construction, it is unimportant that $\AAA$ is  an unlabelled disjoint union of the components $\NN_\alpha$; we could have also built a labelled disjoint union, in which the component $\NN_\alpha$ is defined by a unary predicate indexed by $\alpha$. This shows that it is not important that for nonstandard $\alpha$, $G^{-\alpha}$ does not depend on $\alpha$; what is important is that it does not depend on $G$. 
\

\subsection{The Wehner graph}

We prove Theorem \ref{thm_relativised_Wehner}. As we mentioned above, this is a relativisation of Wehner's proof of Theorem \ref{thm_Slaman_Wehner}. 

For a set $X\in 2^\w$, consider the following family of finite sets:
\[
\FF_X  = \left\{ \{e\}\oplus F\,:\, e< \om, F\subseteq \om\text{ is finite, and } F\neq W_e^X\right\}.
\]
Recall that for a set $A\subseteq \w^2$, for all $n<\w$ we let $A^{[n]} = \left\{x  \,:\, (n,x)\in A \right\}$. For any countable collection $\FF$ of subsets of $\w$, we say that a set $A\subseteq \om^2$ is an {\em enumeration} of $\FF$ if $\FF=\{A^{[n]}:n<\om\}$. We say that a set $Y$ can {\em enumerate} $\F$ is there is an enumeration of $\F$ that is $Y$-c.e.

\begin{proposition} \label{prop_Wehner_family_diagonalizes}
No set $X$ can enumerate $\FF_X$.
\end{proposition}

\begin{proof}
If $X$ were able to enumerate $\FF_X$, then $X$ could compute a function $f$ such that for all $e$, $W_{f(e)}^X$ is a finite set different than $W_e^X$.  This would contradict the recursion theorem.
\end{proof}

\begin{proposition} \label{prop_uniform_Wehner_family}
There is a c.e.\ operator $V$ such that for all $X$ and $Y\nle_\Tur X$, $V(Y,X)$ is an enumeration of $\FF_X$.
\end{proposition}

\begin{proof}
Let $X,Y\in 2^\w$. 
For each $n<\w$ we enumerate $V(Y,X)^{[n]}$ as follows.
Suppose that $n= (e,u,s_0)$.
We start by letting 
\[ V(Y,X)_{s_0}^{[n]} = \{e\}\oplus D_u = \{2e\}\cup \{2x+1: x\in D_u\};\] here $D_u$ is the $u\tth$ finite set of natural numbers. 
At every stage $s>s_0$, if we see that $V(Y,X)_s^{[n]}=\{e\}\oplus W_{e,s}^X$, then we enumerate $2x+1$ into $V(Y,X)_{s+1}^{[n]}$, where $x$ is the least element of $Y\oplus \overline Y$ which is not in $V(Y,X)_s^{[n]}$. Here $\overline Y = \w\setminus Y$. 

Let $\{e\}\oplus F$ be an element of $\FF_X$. There is a stage $s_0$ such that for all $s\ge s_0$, $F\ne W_{e,s}^X$. If $F=D_u$ then $V(Y,X)^{[n]} = \{e\}\oplus F$ where $n = (e,u,s_0)$, as we set $V(Y,X)_{s_0}^{[n]} = \{e\}\oplus F$ and never enumerate other elements into $V(Y,X)^{[n]}$ at any later stage. So $\FF_X\subseteq \{V(Y,X))^{[n]}:n<\om\}$.

Suppose that $Y\nle_\Tur X$.  Pick any $n=(e,u,s_0)$; we claim that $V(Y,X)^{[n]}$ is finite which, by the construction, implies that $V(Y,X)^{[n]}\in \FF_X$. The reason $V(Y,X)^{[n]}$ it is finite is that $Y\oplus \overline Y$ is not c.e.\ in $X$. If there were infinitely many stages at which new elements were enumerated into $V(Y,X)^{[n]}$, then we would end up with 
\[ \{e\} \oplus W_e^X \,=\, V(Y,X)^{[n]} \,=\, \{e\} \oplus (D_u \cup (Y\oplus \overline Y)),\]
whence $Y\oplus \overline Y$ would be c.e.\ in $X$, yielding $Y\le_\Tur X$. 
\end{proof}

Now, following Knight (see \cite{AK00}) and Khoussainov \cite{Bakh_dasies}, we code families of sets in graphs. 

\begin{definition}
Given a set $F\subseteq\om$, we define the \emph{flower graph} $G(F)$ of $F$ as follows:
We start with a vertex $v$, and for each $n\in F$ we add a cycle of length $n+3$ starting and ending in $v$.

Given a family of sets $\FF$, we define the {\em bouquet graph $G(\FF)$} of $\FF$ to be the disjoint union of infinitely many copies of $G(F)$ for each $F\in \FF$, together with infinitely many isolated vertices.
\end{definition}

\begin{lemma}[\cite{AK00,Bakh_dasies}] \label{lem_flower_grpahs_and_enumerations}
For any countable collection $\FF$ of subsets of $\w$, a set $Y$ can compute a copy of $G(\FF)$ if and only if $Y$ can enumerate $\FF$. 

This equivalence is uniform: there is a total Turing functional $\Lambda$ such that for any set $Z$ and any index $e$, if $\FF_e^Z$ is the family enumerated by $W_e^Z$, then $\Lambda(Z,e)$ is a presentation of $G\left( \FF_e^Z\right)$. 
\end{lemma}

For $X\in 2^\w$, we let the \emph{Wehner graph} $G_X$ of $X$ be the bouquet graph $G\left(\FF_X \right)$ of the Wehner family $\FF_X$. Theorem \ref{thm_relativised_Wehner} now follows immediately from Propositions \ref{prop_Wehner_family_diagonalizes} and \ref{prop_uniform_Wehner_family}, and Lemma \ref{lem_flower_grpahs_and_enumerations}.

\

\subsection{Jump inversion for structures}

We set out to prove Theorem \ref{thm_jump_inversion}. The theorem actually follows from an application of an overspill argument to Ash's metatheorem (see \cite{AK00}); the relevant work here is by Ash and Knight, on pairs of structures \cite{AK90}. This work was used in \cite{GHKMMS} to invert the $\alpha$-jump for standard computable ordinals $\alpha$. Ash's theorem has a complicated proof, and its full power is not required to prove jump-inversion for structures. This is why we give a complete proof. The construction we present has its roots in work of Hirschfeldt and White \cite{HW02}. However, Hirschfeldt and White do not give sharp bounds for the pairs of trees they construct, whereas proving Theorem \ref{thm_jump_inversion} does require these sharp bounds. Possibly our construction is not new, but we have not been able to find it in the literature. 

\

The main part of the argument proving Theorem \ref{thm_jump_inversion} is the construction of uniformly computable structures $\SSS_{\alpha,0}$ and $\SSS_{\alpha,1}$, for $\alpha<\delta^*$, which for standard $\alpha < \omck$ discern $\Delta^0_{2\alpha+1}$ predicates, and for nonstandard $\alpha\in \delta^*\setminus \omck$ have the same isomorphism type. This is the content of Propositions \ref{prop_jump_inversion_part_two}, \ref{prop_jump_inversion_part_one}, and \ref{prop_jump_inversion_part_three}. Here discerning $\Delta^0_{2\alpha+1}$ predicates means that the problem of detecting isomorphisms to either $\SSS_{\alpha,0}$ or $\SSS_{\alpha,1}$ is $\Delta^0_{2\alpha+1}$-complete. That is, given a structure $\NN$, $\emptyset^{(2\alpha+1)}$ can determine if $\NN$ is isomorphic to $\SSS_{\alpha,0}$ or $\SSS_{\alpha,1}$; and for any $\Delta^0_{2\alpha+1}$ set $X$, there is a uniformly computable sequence of structures $\seq{\NN_n}_{n<\w}$ such that if $n\in X$, then $\NN_n\cong \SSS_{\alpha,1}$, and if $n\notin X$, then $\NN_n\cong \SSS_{\alpha,0}$. 

\medskip

The building blocks of the structures $\SSS_{\alpha,0}$ and $\SSS_{\alpha,1}$ are fat trees, which we now define.

\

\subsubsection{Fat trees} \label{subsec_fat_trees}

We work with nonempty countable rooted trees of height at most~$\w$. Usually, trees of finite height are defined to be partial orderings $(T,<_T)$ for which for all $y\in T$, the collection of predecessors $\{x\in T\,:\, x<_T y\}$ is linearly ordered and finite; \emph{rooted} means that $T$ has a $<_T$-least element, named $\rooot(T)$. However, such a presentation of a tree $T$ does not allow us to effectively compute the parent (the immediate predecessor) of a non-root element of $T$. Further, under this definition, a homomorphism $f\colon T\to S$ of trees does not need to take immediate successors to immediate successors, or the root to the root. To overcome this, when we consider them as structures, we add the parent function which maps every non-root element of $T$ to its parent, and the root to itself. Note that the ordering $<_T$ is $\LL_{\w_1,\w}$-definable from the parent function, and so we may omit it when specifying a tree. 

Each such tree is isomorphic to a downward closed subset of $\w^{<\w}$ (the collection of all finite strings of natural numbers), with the ordering given by extension; in other words, the parent unary function is interpreted as the function which chops off the last bit of the string. However, it will be useful to use the general notion; in particular, we will use subsets of $\w^{<\w}$ which do not necessarily contain the empty string~$\seq{}$, but which are trees under the ordering of string extension. 

\medskip

Let $T\in H$ be a tree, and suppose that $H$ believes that $T$ is well-founded. That is, $H$ does not contain an infinite path of $T$. Then $H$ contains a unique \emph{rank function} for $T$: a function $r$ from $T$ to the countable ordinals of $H$, such that for all $x\in T$,  $r(x)$ is the least upper bound (in $H$) of $r(y)+1$, for $y>_T x$. We let $\rk_T$ be this unique rank function, and let $\rk(T) = \rk_T(\rooot(T))$.\footnote{Note that this notation differs by 1 from the standard set-theoretic definition, which lets $\rk(T) = \rk(\rooot(T))+1$.} The range of $\rk_T$ contains every $\alpha<\rk(T)$, in other words, $\range \rk_T = \{ \beta\,:\, \beta\le \rk(T)\} = \rk(T)+1$.

For any tree $T$ and $x\in T$, let $T_x = \{ y\in T\,:\, y\ge_T x\}$ with the partial ordering restricted from $<_T$; so $x = \rooot(T_x)$. We have $\rk_{T_x} = \rk_T \rest{T_x}$, so $\rk(T_x) = \rk_T(x)$. We also let $\hht_T(X)$, the \emph{height} of $x$ on $T$, be the size of the set $\{y\in T\,:\, y<_T x\}$. The height $\hht(T)$ of a tree $T$ is the supremum of the height of its elements. If $T$ has finite height, then $\rk(T)= \hht(T)$. If the height of $T$ is $\w$, then $\rk(T)$ is infinite. For any ranked tree $T$, for all $k\le \hht(T)$, $\rk(T) = \sup \{\rk_T(x)+k\,:\, \hht_T(x)=k\}$.  

If $T$ is a hyperarithmetic well-founded tree, then the rank function of $T$ is hyperarithmetic (and the rank of $T$ is a computable ordinal). It follows that the rank function of $T$ (as computed in $V$) is an element of $H$; by $\Delta^1_1$ absoluteness between $H$ and $V$, we see that $\rk_T$ is the ``true'' rank function of $T$, and that $\rk(T)$ is computed correctly in $H$. In fact, for a hyperarithmetic tree $T$ which $H$ believes is well-founded, $\rk(T)<\omck$ if and only if $T$ is well-founded, because if $\rk(T)<\omck$, then the well-foundedness of $\rk(T)$ shows that $T$ is well-founded. 

\begin{definition}\label{def_fat_tree}
	Let $T$ be a tree which $H$ believes is well-founded. The tree $T$ is \emph{fat} if for all $x\in T$, for all $\gamma < \rk_T(x)$, there are infinitely many children (immediate successors) $y$ of $x$ on $T$ such that $\rk_T(y) = \gamma$. 
\end{definition}

If $T$ is truly well-founded, then the fatness of $T$ does not depend on the choice of $H$. The utility of fat trees is that they are universal for their rank. 

\begin{proposition}\label{prop_fat_trees_universal}
	Let $S$ and $T$ be trees which $H$ believes are well-founded. If $\rk(S)\le \rk(T)$, and $T$ is fat, then $S$ is embeddable into $T$. If both $T$ and $S$ are fat and $\rk(S) = \rk(T)$, then $S$ and $T$ are isomorphic. 
	
	Furthermore, if both $T$ and $S$ are fat and are ill-founded, then they are isomorphic, regardless of their nonstandard rank. 
\end{proposition}

\begin{proof}
	An isomorphism $f\colon S\to T$ is built by a back-and-forth argument, building an isomorphism from the roots up. Inductively, we map an element $x\in S$  to $f(x)\in T$ such that $\rk_S(x)< \omck$ if and only if $\rk_T(f(x))< \omck$, and if so, then these ranks are equal. Fatness, together with the fact that $\omck$ is the well-founded part of the ordinals of $H$, shows that the induction can always continue. 
	
	The embedding is similarly built, by a forth-only argument. 
\end{proof}

Given any $\alpha<\delta^*$, we can effectively obtain a computable (index for a) fat tree of rank $\alpha$. This is done using the \emph{fattening} operation on trees. For any tree $T$, let $\fat(T)$ be the subset of $\w^{<\w}$ which consists of the (nonempty) sequences of the form $\seq{\rooot(T),(x_1,n_1),(x_2,n_2),\dots, (x_k,n_k)}$, where $k<\w$, $n_1,n_2,\dots, n_k\in \Nat$ and $\rooot(T)<_T x_1 <_T x_2 <_T \dots <_T x_k$. It is easy to see that the fattening operation induces a computable map on indices of computable trees. 

\begin{lemma}\label{lem_fattening_makes_fat}
	If $T$ is well-founded in $H$, then $\fat(T)$ is fat, and $\rk(\fat(T))= \rk(T)$. 
\end{lemma}

\begin{proof}
	The function mapping $\seq{\rooot(T),(x_1,n_1),(x_2,n_2),\dots, (x_k,n_k)}$ to $\rk_T(x_k)$ (and $\rooot(\fat(T)) = \seq{\rooot(T)}$ to $\rk(T)$) is the rank function for $\fat(T)$. The fact that for $x\in T$, the range of $\rk_{T_x}$ is $\rk_T(x)+1$ shows that $\fat(T)$ is fat. 
\end{proof}

For $\alpha<\delta^*$, let $S_{\alpha}$ be the \emph{tree of descending sequences} from $\alpha$; 
\[ S_\alpha = \left\{\seq{\alpha_1,\alpha_2,\dots, \alpha_k} \,:\, \alpha>\alpha_1>\alpha_2 > \dots \alpha_k \right\};\]
the root of $S_{\alpha}$ is the empty sequence $\seq{}$. It is easy to check that for all nonempty 
$\seq{\alpha_1,\dots, \alpha_k}\in S_{\alpha}$, 
$\rk(\seq{\alpha_1,\dots, \alpha_k}) = \alpha_k$, and so that  
$\rk(S_\alpha) = \alpha$; this is because for all $\alpha< \delta^*$, $\alpha = \sup\{ \beta+1\,:\, \beta<\alpha\}$. We let $T_\alpha = \fat(S_\alpha)$; so $T_\alpha$ is a computable fat tree of rank $\alpha$, 
and a computable index for $T_\alpha$ is effectively obtained from $\alpha$. 

We let $T_\infty$ be a fat, ill-founded computable tree. So for all $\alpha\in \delta^*\setminus \omck$, $T_\alpha\cong T_\infty$. 

\medskip

\subsubsection{The adjusted hyperarithmetic hierarchy} \label{subsubsec_corrected}

Before we define the structures $\SSS_{\alpha,0}$ and $\SSS_{\alpha,1}$, we explain why we modified the iterated jump and used the sets $Z_{(\alpha)}$ rather than $Z^{(\alpha)}$ (Subsection \ref{subsubsec_jump_inversion}), that is, why we work with $Z^{(\alpha+1)}$ rather than $Z^{(\alpha)}$ if $\alpha$ is infinite, but not if it is finite. 

Recall the following definition of the hyperarithmetic hierarchy. For each $\alpha<\omck$, we fix an effective 
listing $\seq{W_{e,\alpha}}_{e<\w}$ of the $\Sigma^0_\alpha$ subsets of $\Nat$. We let $W_{e,1} = W_e$ be the 
$e\tth$ c.e.\ set (or we could start with $\alpha=0$, if we like, by listing all the primitive recursive sets, 
say). Given $W_{i,\beta}$ for all $i<\w$ and $\beta<\alpha$, we let $W_{e,\alpha}$ be the union of the sets 
$\overline{W_{i,\beta}}$, where $(i,\beta)\in W_e \cap (\w\times \alpha)$. 

Note that in fact, this definition can be pursued in $H$, as we did with ranks of $H$-well-founded trees; this gives us an effective listing, for each $\alpha<\delta^*$, of what $H$ defines as the $\Sigma^0_\alpha$ sets. These definitions are relativised naturally to any oracle $Z\in H$, and for $\alpha<\omck$, to any oracle $Z\in 2^\w$. 

Now there is a discrepancy between the finite and infinite levels of the hierarchy, in the relationship between $\Sigma^0_\alpha$ and $\emptyset^{(\alpha)}$. For nonzero $n<\w$, $\emptyset^{(n)}$ is $\Sigma^0_{n}$-complete (for many-one reductions), but for computable successor ordinals $\alpha\ge \w$, $\emptyset^{(\alpha)}$ is merely $\Sigma^0_{\alpha-1}$-complete. This is because for limit ordinals $\alpha$, $\Sigma^0_\alpha$ sets are effective unions of $\Pi^0_\beta$ sets for $\beta<\alpha$ unbounded in $\alpha$. Similarly, for $n<\w$, a set is $\Sigma^0_{n+1}$ if and only if it is c.e.\ in $\emptyset^{(n)}$; whereas for all $\alpha\ge \w$, a set is $\Sigma^{0}_{\alpha}$ if and only if it is c.e. in $\emptyset^{(\alpha)}$. 

We thus use the modification employed by Ash and Knight in \cite{AK00} to overcome this split between the finite and infinite case. For all $Z\in H\cap 2^\w$, for all $\alpha<\delta^*$, a set is $\Sigma^0_{\alpha+1}(Z)$ if and only if it is c.e.\ in $Z_{(\alpha)}$, and $Z_{(\alpha)}$ is $\Sigma^0_\alpha(Z)$-complete. The same holds for $\alpha<\omck$ for all $Z\in 2^\w$. 

All of these equivalences are uniform. For example, there is a c.e.\ operator $\Gamma$ such that for all $Z\in 2^\w$, all $\alpha<\omck$ and all $e<\w$, $\Gamma(Z_{(\alpha)},\alpha,e) = W_{e,\alpha+1}(Z)$.

\subsubsection{The structures $\SSS_{\alpha,0}$ and $\SSS_{\alpha,1}$}

We use the fat trees $T_\alpha$ to define the structures $\SSS_{\alpha,0}$ and $\SSS_{\alpha,1}$. Both will be \emph{pairs of trees}, or labelled disjoint unions of trees. For trees $S$ and $T$, the universe of the structure $(S,T)$ is the disjoint union of $S$ and $T$; the parent function is defined on both parts; and a unary predicate defines $S$ in the structure. 

\begin{definition}\label{def_alpha_complete_pair}
	For nonzero $\alpha< \delta^*$, let $\SSS_{\alpha,0} = (T_{\w\alpha}, T_{\w\alpha+1})$, and $\SSS_{\alpha,1} = (T_{\w\alpha+1}, T_{\w\alpha})$. 
\end{definition}

\begin{proposition} \label{prop_jump_inversion_part_two}
	For $\alpha,\beta\in \delta^*\setminus \omck$, $\SSS_{\alpha,0}\cong \SSS_{\beta,0}\cong \SSS_{\beta,1}$.
\end{proposition}

\begin{proof}
	All three structures are isomorphic to $(T_\infty,T_\infty)$. 
\end{proof}

The following proposition states that the isomorphism problem for the pair $(\SSS_{\alpha,0},\SSS_{\alpha,1})$ is $\emptyset_{(2\alpha)}$-computable, uniformly in $\alpha$, in a relativisable way. Note that this problem is not trivial: for $\alpha<\omck$, $\SSS_{\alpha,0}\ncong \SSS_{\alpha,1}$, because $T_{\w\alpha}\ncong T_{\w\alpha+1}$. This relies on the fact that a standard ordinal cannot be embedded into a smaller ordinal. This shows that the isomorphism between say $\SSS_{\alpha,0}$ and $\SSS_{\alpha,1}$ for nonstandard $\alpha$ cannot belong to $H$. 

\medskip

Let $\seq{\Phi_e}$ be an effective sequence of all Turing functionals. 

\begin{proposition}\label{prop_jump_inversion_part_one}
	There is a Turing functional $\Theta$ such that for all $Z\in 2^\w$, for all nonzero $\alpha<\omck$, and all $e<\w$,
	\begin{itemize}
		\item If $\Phi_e(Z)\cong \SSS_{\alpha,0}$, then $\Theta(Z_{(2\alpha)},\alpha,e) = 0$; and 
		\item If $\Phi_e(Z)\cong \SSS_{\alpha,1}$, then $\Theta(Z_{(2\alpha)},\alpha,e) = 1$.
	\end{itemize}
\end{proposition}

\begin{proof}
	We analyse the complexity, for $\alpha<\omck$, of the class of well-founded trees whose rank is bounded by $\alpha$. We need to work in relativised form. For any $Z\in 2^\w$ and $\alpha<\omck$, let $R_\alpha(Z)$ be the collection of indices $e<\w$ such that $\Phi_e(Z)$ is total and is a tree whose rank is smaller than $\alpha$. 
	
	We show that for all computable $\alpha\ge 1$ and $Z\in 2^\w$:
	\begin{enumerate}
		\item $R_{\w\alpha}(Z)\le_\Tur Z_{(2\alpha)}$; and
		\item for all $n<\w$, $R_{\w\alpha+n}(Z)$ is co-c.e.\ in $Z_{(2\alpha)}$.
	\end{enumerate}
	These are proved simultaneously by effective transfinite recursion on $\omck$ (that is, on the well-founded initial segment of the computable well-ordering $\delta^*$). That is, by recursion on $\alpha<\omck$, we define Turing functionals which given $Z$, $\alpha$ and $n$, produce a $Z_{(2\alpha)}$-index for $R_{\w\alpha}(Z)$ (an index $e$ such that $\Phi_e(Z_{(2\alpha)}) = R_{\w\alpha}(Z)$), and a $\Pi^0_{2\alpha+1}(Z)$-index for $R_{\w\alpha+n}(Z)$ (an index $e$ such that ${W_{e,2\alpha+1}(Z)} = \w\setminus R_{\w\alpha+n}(Z)$), recalling that a set is co-c.e.\ in $Z_{(2\alpha)}$ if and only if it is $\Pi^0_{2\alpha+1}(Z)$. 
	
	Since a tree has a finite rank if and only if it has finite height, $R_\w(Z)$ is $Z_{(2)} = Z''$-computable, as the condition that $\Phi_e(Z)$ is total is $\Pi^0_2(Z)$, and the condition that $\Phi_e(Z)$ has finite height is $\Sigma^0_2(Z)$. For $\alpha>1$, given (2) for $\beta<\alpha$, we in fact see that $R_{\w\alpha}(Z)$ is $\Sigma^0_{2\alpha}$, because for any tree $T$, $\rk(T)<\w\alpha$ if and only if there is some $\beta<\alpha$ and $n<\w$ such that $\rk(T)< \w\beta+n$. We then use the fact that $Z_{(2\alpha)}$ is $\Sigma^0_{2\alpha}$-complete to see that $Z_{(2\alpha)}$ computes $R_{\w\alpha}(Z)$. 
	
	Now for any $\alpha\ge 1$, we see that for any $n<\w$, $R_{\w\alpha+n}(Z)$ is co-c.e.\ in $R_{\w\alpha}(Z)$, and then use (1). This is because for any tree $T$, $\rk(T)<\w\alpha+n$ if and only if for every $x\in T$ of height $n$, $\rk(T_x)< \w\alpha$. 
	
	\medskip
	
	Now we define $\Theta$ as follows. Given $Z\in 2^\w$, $\alpha<\delta^*$ and $e<\w$, if $\Phi_e(Z) = (S,T)$, then we run the co-enumeration of $R_{\w\alpha+1}(Z)$ with oracle $Z_{(2\alpha)}$. We let $\Theta(Z_{(2\alpha)},\alpha,e) = 1$ if we first notice that $\rk(S)>\w\alpha$, and output 0 if we first notice that $\rk(T)> \w\alpha$. Note, of course, that given $Z_{(2\alpha)}$ and $\alpha$, we can effectively find~$Z$; and that from a $Z$-index for $(S,T)$ we can effectively find $Z$-indices for $S$ and for~$T$. 
\end{proof}

\subsubsection{Hardness of the isomorphism problem}

We want to code any $\Delta^0_{2\alpha+1}$ set into an isomorphism problem for the pair $\SSS_{\alpha,0}$ and $\SSS_{\alpha,1}$. We start with a recursive definition of the sets $\emptyset_{(\alpha)}$ for $\alpha<\delta^*$. 

\

Recall that we also consider $\delta^*$ as a notation in $\OO^H$. This means that for limit $\alpha<\delta^*$ we uniformly obtain a computable and increasing sequence $\seq{\alpha_s}_{s<\w}$ which is cofinal in $\alpha$. We may assume that for all $s$, $\seq{\alpha_s}$ is odd. 

For a successor $\alpha<\delta^*$, we let $\alpha_s = \alpha-1$ for all $s$. 

\begin{lemma}\label{lem_definition_of_H_sets}
	There is a computable function $f$ such that for all $\alpha<\delta^*\setminus \{0,1\}$ and all $m\in \Nat$, 
	$m\in \emptyset_{(\alpha)}$ if and only if for some $s<\w$, $f(\alpha, m, s) \notin \emptyset_{(\alpha_s)}.$
	
	Furthermore, we may assume that if $f(\alpha,m,s)\notin \emptyset_{(\alpha_s)}$, then for all $t\ge s$, $f(\alpha,m,t)\notin \emptyset_{(\alpha_t)}$; and that $f(\alpha,m,s)\ge s$. 
\end{lemma}

Hence, either $m\notin \emptyset_{(\alpha)}$, and for all $s$, $f(\alpha,m,s)\in \emptyset_{(\alpha_s)}$; or $m\in \emptyset_{(\alpha)}$, and for all but finitely many $s$, $f(\alpha,m,s)\notin \emptyset_{(\alpha_s)}$. 

\begin{proof}
	Let $\alpha\ge 1$. Since $\emptyset_{(\alpha)}$ is $\Sigma^0_\alpha$, we can find, effectively in $\alpha$ and $m$, a sequence $\seq{e(t),\beta(t)}$ (where $\beta(t)<\alpha$) such that $m\in \emptyset_{(\alpha)}$ if and only if there is some $t<\w$ such that $m\notin W_{e(t),\beta(t)}$. For all $t$ we can effectively find some $s=s(\alpha,m,t)$ such that $\alpha_s\ge \beta(t')$ for all $t'\le t$. Since $\emptyset_{(\alpha_s)}$ is $\Sigma^0_{\alpha_s}$-complete, this means that we can find a function $f$ such that $f(\alpha,m,s)\in \emptyset_{(\alpha_s)}$ if and only if for all $t'\le t$, $m\in W_{e(t'),\beta(t')}$. We get $f(\alpha,m,s)\ge s$ by using the padding lemma. 
\end{proof}

Using the fattening operation we discussed in Subsection \ref{subsec_fat_trees}, we define two operations on countable sequences trees, mapping (indices of) uniformly computable sequences of trees to (indices of) computable trees. Let $\seq{T_i}_{i<\w}$ be a sequence of trees. We let
\begin{itemize}
	\item $\supr_i T_i = \fat(S)$, where $S$ is obtained from the disjoint union of the trees $T_i$ and identifying their roots into a single root. We also let
	\item $\mini_i T_i = \fat(S)$, where $S$ is the collection of all nonempty strings of the form 
	\[ \seq{ \seq{x_{0,0}}, \seq{x_{1,0}, x_{1,1}}, \seq{x_{2,0}, x_{2,1}, x_{2,2}}, \dots, \seq{x_{k,0}, x_{k,1}, \dots, x_{k,k}} },\]
	where for each $j\le k$ and $i\le j$, $x_{j,i}\in T_i$, $x_{i,i} = \rooot(T_i)$, and $x_{i,i}<_{T_i} x_{i+1,i} <_{T_i} \dots <_{T_i} x_{k,i}$. The root of $S$ is $\seq{\seq{\rooot(T_0)}}$. 
\end{itemize}

\begin{lemma}\label{lem_supr_and_mini}
	Suppose that $\seq{T_i}\in H$ and that every $T_i$ is well-founded in $H$. Then
	\begin{enumerate}
		\item $\rk(\supr_i T_i) = \sup_i \rk(T_i)$; and
		\item $\rk(\mini_i T_i) = \min_i (\rk(T_i)+i)$.
	\end{enumerate}
\end{lemma}

\begin{proof} 
	Work in $H$. 
	
	If $\supr_i T_i = \fat(S)$ with $S$ defined as above, then for all $i$ and $x\in T_i\setminus \{\rooot(T_i)\}$, $\rk_S(x) = \rk_{T_i}(x)$. So 
	\[ \rk(S) = \sup_{i<\w} \sup_{x\in T_i\setminus \{\rooot(T_i)\}} (\rk_{T_i}(x)+1)  = \sup_{i<\w} \rk_{T_i}(\rooot(T_i)) = \sup_i \rk(T_i) .\]
	
	\medskip
	
	Let $\mini_i T_i = \fat(S)$ with $S$ defined as above. Suppose that $\bar \s = \seq{\s_0,\s_2,\dots, \s_k}\in S$. Let $\s_k = \seq{x_0,x_1,\dots, x_k}$; so $x_k = \rooot(T_k)$. Then $\rk_S(\bar \s)\le \rk(T_k)$. To see this, consider the (non-injective) homomorphism $f$ from $S_{\bar \s}$ to $T_k$, mapping $\seq{\s_1,\s_2,\dots, \s_{m-1},\seq{y_0,y_1,\dots, y_m}}$ to $y_k$. We see by induction on $\rk(\bar \tau)$ that for all $\bar \tau\in S_{\bar \s}$, $\rk_S(\bar \tau) \le \rk_{T_k}(f(\bar \tau)))$. So for all $\bar \s\in S$ of height $k$, we have $\rk_S(\bar s)\le \rk(T_k)$. If $S$ contains elements of height $k$, this shows that $\rk(S) \le \rk(T_k)+k$. Otherwise, $\rk(S) \le k$, so certainly $\rk(S)\le \rk(T_k)+k$. 
	
		Now let $k<\w$ such that for all $i<\w$, $\rk(T_k)+k\le \rk(T_i)+i$. We show that $T_k$ is embeddable into $S_{\bar \s}$ for some $\bar \s\in S$ of height $k$ in $S$, whence we get that $\rk(S)\le \rk(T_k) + k$. 
	
	For $i<k$, since $\rk(T_i)\ge \rk(T_k)+ (k-i)$, we can find a sequence $\rooot(T_i) = x_{i,i}<_{T_i} x_{i+1,i} <_{T_i} \dots <_{T_i} x_{k,i}$ such that $\rk_{T_i}(x_{k,i}) \ge \rk(T_k)$. We let 
		\[ \bar \s = \seq{ \seq{x_{0,0}}, \seq{x_{1,0}, x_{1,1}}, \seq{x_{2,0}, x_{2,1}, x_{2,2}}, \dots, \seq{x_{k,0}, x_{k,1}, \dots, x_{k,k-1}, \rooot(T_k)} },\]
	and set $f(\rooot(T_k)) = \bar \s$. By induction on the height of $y\in T_k$, we define $f(y)$ so that $f\colon T_k\to S_{\bar \s}$ is an embedding. Let $x\in T_k$ have height $m$ in $T_k$. Suppose that $f(x)$ has been defined, and that $f(x)  = \seq{\tau_0,\tau_1,\dots, \tau_{k+m}}$ with $\tau_{k+m} = \seq{x_0,x_1,\dots, x_{k+m}}$ such that $x_{i}\in T_i$, $\rk_{T_i}(x_i) \ge \rk_{T_k}(x)$, and $x_k = x$. Let $y$ be a child of $x$ on $T_k$. For all $i\le k+m$, we can find some $y_i >_{T_i} x_i$ such that $\rk_{T_i}(y_i) \ge \rk_{T_k}(y)$; we also choose $y_k=y$. For $i= k+m+1$, since $\rk_{T_i} + (i-k)\ge \rk(T_k)$, we have $\rk(T_k)\le \rk(T_i) + (m+1)$. Since the height of $y$ in $T_k$ is $m+1$, $\rk_{T_k}(y) + (m+1) \le \rk(T_k)$. Putting these together, we get $\rk_{T_k}(y)\le \rk(T_i)$, so $\rk_{T_i}(\rooot(T_i))\ge \rk_{T_k}(y)$; we choose $y_{i} = \rooot(T_i)$. We then let $f(y) = \seq{\tau_0,\tau_1,\dots, \tau_{k+m}, \seq{y_0,y_1,\dots, y_{k+m+1}}}$.
\end{proof}

\

We define computable trees $T^\beta_\gamma(m)$ for all nonzero $\gamma<\delta^*$, all $m\in \Nat$, and all $\beta < \delta^*$. 
An index for $T^\beta_\gamma(m)$ is obtained effectively from $\beta,\gamma$ and $m$. These trees are defined by working in $H$, performing effective transfinite recursion on $\gamma<\delta^*$. There are three cases:
\begin{itemize}
	\item $\gamma=1$: if $m\notin \emptyset'$, then we let $T^\beta_\gamma(m)$ consist of a root only. If at stage $s$ we see that $m\in \emptyset'$, then we let $T^\beta_\gamma(m)\cong T_\beta$, using elements with G\"odel numbers greater than $s$. 

	\item $\gamma>1$ is even: we let $T^\beta_\gamma(m) = \mini_{s\ge m} T^\beta_{\gamma_s} (f(\gamma,m,s))$.

	\item $\gamma>1$ is odd: we let $T^\beta_\gamma(m) = \supr_{s<\w}T^{\beta}_{\gamma_s}(f(\gamma,m,s))$. Of course in this case, $\gamma_s = \gamma-1$.
\end{itemize}
Of course, here $f$ is the function guaranteed by Lemma \ref{lem_definition_of_H_sets}. 
Note that for all $\beta,\gamma$ and $m$, $T^\beta_\gamma(m)$ is fat. We calculate ranks. 

\begin{proposition}\label{prop_calculate_ranks}
	Let $\gamma\in \delta^*\setminus\{0\}$, and let $m\in \Nat$. 
	\begin{enumerate}
		\item If $\gamma=2\delta$ is even, then for all $\beta\ge \w\delta$,
			\begin{enumerate}
				\item if $m\notin \emptyset_{(\gamma)}$, then $\rk(T^\beta_\gamma(m)) = \beta$;
				\item if $m\in \emptyset_{(\gamma)}$, then $\rk(T^\beta_\gamma(m)) < \w\delta$. 
			\end{enumerate}
			\item If $\gamma=2\delta+1$ is odd, then for all $\beta\ge \w\delta$,
				\begin{enumerate}
					\item if $m\notin \emptyset_{(\gamma)}$, then $\rk(T^\beta_\gamma(m)) = \w\delta$; 
					\item if $m\in \emptyset_{(\gamma)}$, then $\rk(T^\beta_\gamma(m)) = \beta$. 
				\end{enumerate}
	\end{enumerate}
\end{proposition}

\begin{proof}
	Working in $H$, we verify the proposition by induction on $\gamma$. 
	
	For $\gamma=1$, we have $\delta=0$. If $m\notin \emptyset_{(1)} = \emptyset'$, then $T^\beta_\gamma(m)$ has only a root, and its rank is $0= \w\delta$. If $m\in \emptyset'$ then $T^\beta_\gamma(m)\cong T_\beta$, whose rank is $\beta$. 
	
	\medskip
	
	Let $\gamma = 2\delta$ be even. For $s<\w$, let $S_s = T^\beta_{\gamma_s}(f(\gamma,m,s))$; so $T^\beta_\gamma(m) = \mini_{s\ge m} S_s$, whence $\rk(T^\beta_\gamma(m)) = \min_{s<\w} (\rk(S_s)+(s-m))$. For each $s$, $\gamma_s$ is odd: if $\gamma$ is a limit, then we required that $\gamma_s$ be odd; and if $\gamma$ is a successor, then $\gamma_s = \gamma-1$ is odd. Let $\delta_s = \integer{\gamma_s/2}$, so $\gamma_s = 2\delta_s +1$. For all $s$, $\delta_s<\delta$ because $\gamma_s < \gamma$. Hence $\beta \ge \w\delta_s$. So by induction, $\rk(S_s) = \w\delta_s$ if $f(\gamma,m,s)\notin \emptyset_{(\gamma_s)}$, and $\rk(S_s) = \beta$ if $f(\gamma,m,s)\in\emptyset_{(\gamma_s)}$. 
	
	If $m\notin \emptyset_{(\gamma)}$, then for all $s$, $f(m,\gamma,s)\in \emptyset_{(\gamma_s)}$, so for all $s$, $\rk(S_s) = \beta$. Then $\rk(T^\beta_\gamma(m)) = \min_{s\ge m} (\beta + (s-m)) = \beta$. If $m\in \emptyset_{(\gamma)}$, then there is some $t<\w$ such that for all $s<t$, $f(\gamma,m,s)\in \emptyset_{(\gamma_s)}$, and for all $s\ge t$, $f(\gamma,m,s)\notin \emptyset_{(\gamma_s)}$. Hence, for all $s<t$, $\rk(S_s) = \beta$, and for all $s\ge t$, $\rk(S_s) = \w\delta_s$. Since $\delta_s<\delta$, we have $\w\delta_s+s < \w\delta \le \beta$. It follows that $\rk(T^\beta_\gamma(m)) = \w\delta_{\max\{t,m\}} + \max\{t,m\} - m < \w\delta$. 
	
	Before we check the odd case, we note that $\rk(T^\beta_\gamma(m))\ge \w\delta_m$; this, because $\seq{\delta_s}$ is non-decreasing.

	\medskip
	
	Now let $\gamma = 2\delta+1$ be odd, so $\gamma-1 = 2\delta$, and for all $s$, $\gamma_s = \gamma-1$. Again let $S_s = T^\beta_{\gamma_s}(f(\gamma,m,s))$. In this case $T^\beta_\gamma(m) = \supr_{s<\w} S_s$, and so $\rk(T^\beta_\gamma)(m) = \sup_{s<\w}\rk(S_s)$. Noting that $\gamma-1$ is even, induction shows that if $f(\gamma,m,s)\notin \emptyset_{(\gamma-1)}$, then $\rk(S_s) =\beta(s)$, and if $f(\gamma,m,s)\in \emptyset_{(\gamma-1)}$, then $\rk(S_s) < \w\delta$. 
	
	If $m\in \emptyset_{(\gamma)}$, then for all but finitely many $s<\w$ we have $f(\gamma,m,s)\notin \emptyset_{(\gamma-1)}$; so for all but finitely many $s$, $\rk(S_s) = \beta$; for other $s$, we have $\rk(S_s) < \w\delta \le \beta$. In this case, $\rk(T^\beta_\gamma(m)) = \sup_s \rk(S_s) = \beta$.
	
	If $m\notin \emptyset_{(\gamma)}$, then for all $s$, $f(\gamma,m,s)\in \emptyset_{(\gamma-1)}$, so for all $s$, $\rk(S_s)< \w\delta$, so $\rk(T^\beta_\gamma(m)) \le \w\delta$. However, we checked that $\rk(S_s) \ge \w\delta_{f(\gamma,m,s)}$, where $\delta_t = \integer{(\gamma-1)_t /2}$. As $f(\gamma,m,s)\ge s$, $\rk(S_s)\ge \w\delta_s$. Since $\sup_s \delta_s = \delta$, we have $\rk(T^\beta_\gamma(m))\ge \w\delta$. 	
\end{proof}

\

We can now finally show the $\Delta^0_{2\alpha+1}$-hardness of the isomorphism problem for the pair $(\SSS_{\alpha,0},\SSS_{\alpha,1})$. 

\begin{lemma}\label{lem_hardness_ce}
	Let $\alpha\in \delta^*\setminus\{0\}$. If $A\le_\Tur \emptyset_{(2\alpha)}$, then there is a uniformly computable sequence of structures $\seq{\NN_{n}}_{n<\w}$ such that for all $n$, 
	\begin{itemize}
		\item if $n\in A$, then $\NN_n\cong \SSS_{\alpha,1}$; and 
		\item if $n\notin A$, then $\NN_n\cong \SSS_{\alpha,0}$.
	\end{itemize}
\end{lemma}

\begin{proof}
	Let $g\colon \w\to \w$ be a many-one reduction of $A$ to $\emptyset_{(2\alpha+1)}$, and $h$ be a many-one reduction of $\w\setminus A$ to $\emptyset_{(2\alpha+1)}$. We then let 
	\[ \NN_n = \left(T^{\w\alpha+1}_{2\alpha+1}(g(n)) , T^{\w\alpha+1}_{2\alpha+1}(h(n)) \right) . \qedhere\]
\end{proof}

In fact, this hardness is uniform, given $\alpha$ and the many-one reductions of $A$ and its complement to $\emptyset_{(2\alpha+1)}$. Moreover, it can be relativised.

\begin{proposition}\label{prop_jump_inversion_part_three}
	For any total Turing functional $\Phi$ there is a Turing functional $\Psi$ such that for any set $X$ and any nonzero $\alpha<\delta^*$, for all $n$, 
	\begin{itemize}
		\item if $\Phi(X,\emptyset_{(2\alpha)},n) = 1$, then $\Psi(X,\alpha,n)\cong \SSS_{\alpha,1}$; and		
		\item if $\Phi(X,\emptyset_{(2\alpha)},n) = 0$, then $\Psi(X,\alpha,n)\cong \SSS_{\alpha,0}$.
	\end{itemize}
\end{proposition}

\begin{proof}
	We may assume that for all $X$, $Y$ and $n$, $\Phi(X,Y,n)\in \{0,1\}$. The compactness of $2^\w$ shows that we can regard $\Phi$ as a \emph{truth-table} functional (this is Nerode's theorem \cite{Nerode:57}). There is a computable sequence of pairs of (canonical indices for) clopen subsets $\CC_n$ and $\DD_n$ of $2^\w$, such that for all $X$ and $Y$, $\Phi(X,Y,n)=1$ if and only if $X\in \CC_n$ and $Y\in \DD_n$; otherwise $\Phi(X,Y,n)=0$. 
	
	The set of $n<\w$ such that $\emptyset_{(2\alpha)}\in \DD_n$ is computable in $\emptyset_{(2\alpha)}$, uniformly in~$\alpha$. By the uniform version of Lemma \ref{lem_hardness_ce}, there is a uniformly computable array $\seq{\NN_{\alpha,n}}_{\alpha<\delta^*,n<\w}$ of structures such that for all nonzero $\alpha<\delta^*$ and all $n<\w$, 
	\begin{itemize}
		\item if $\emptyset_{(2\alpha)}\in \DD_n$, then $\NN_{\alpha,n}\cong \SSS_{\alpha,1}$; and
		\item if $\emptyset_{(2\alpha)}\notin \DD_n$, then $\NN_{\alpha,n}\cong \SSS_{\alpha,0}$. 		
	\end{itemize}
	The functional $\Psi$ is now defined as follows: given $X$, $\alpha>0$ and $n$, if $X\notin \CC_n$, then we output $\SSS_{\alpha,0}$; if $X\in \CC_n$, then we output $\NN_{\alpha,n}$. 
\end{proof}

\subsubsection{The proof of Theorem \ref{thm_jump_inversion}}

Let $\alpha<\delta^*$, and let $G$ be a graph. If $\alpha=0$, then we let $G^{-0}=G$. If $\alpha>0$, then we let $G^{-\alpha}$ be the structure obtained from $G$ by replacing every edge by a copy of $\SSS_{\alpha,1}$, and every non-edge by a copy of $\SSS_{\alpha,0}$. As we mentioned above, a similar construction in \cite{GHKMMS} uses pairs built from linear orderings of the form $\Int^\alpha$ instead of $\SSS_{\alpha,0}$ and $\SSS_{\alpha,1}$. 

Formally, a unary predicate $V$ defines in $G^{-\alpha}$ the set of vertices of $G$. A ternary predicate $D$ defines a partition of $G^{-\alpha}\setminus V$, the classes of which are indexed by pairs of elements of $V$. Adding the language of $\SSS_{\alpha,i}$, for $a,b\in V$, $D(a,b,-)$ is the domain of either $\SSS_{\alpha,0}$ or $\SSS_{\alpha,1}$, depending on whether the edge $(a,b)$ is in $G$ or not.  

\medskip

To prove part (1) of Theorem \ref{thm_jump_inversion}, we need to show, for all nonzero $\alpha<\omck$, that if a set $X$ computes a copy of $G^{-\alpha}$, then $X_{(2\alpha)}$ computes a copy of $G$. The isomorphic copies of $G^{-\alpha}$ are the structures $H^{-\alpha}$ for $H\cong G$; so it suffices to show that if a set $X$ computes $G^{-\alpha}$, then $X_{(2\alpha)}$ computes $G$. 

Say $X$ computes $G^{-\alpha}$. Then $X$ computes the set of vertices $V$ of $G$. To recover the edges of $G$, for $a,b\in V$, $X_{(2\alpha)}$ examines the structure $M_{a,b}$ in $G^{-\alpha}$ whose domain is $D(a,b,-)$; an $X$-computable index for $M_{a,b}$ is effectively obtained. Using the functional given by Proposition \ref{prop_jump_inversion_part_one}, $X_{(2\alpha)}$ can tell if $M_{a,b}$ is isomorphic to $\SSS_{\alpha,0}$ or $\SSS_{\alpha,1}$, and so decide if $(a,b)$ is an edge of $G$ or not. 

We note that this process can be reversed; say $X_{(2\alpha)}$ computes $G$. Then using a relativisation of Lemma \ref{lem_hardness_ce} to $X$, which is possible for standard $\alpha<\omck$, we see that we can indeed $X$-computably replace the edges of $G$ by the correct copies of either $\SSS_{\alpha,0}$ or $\SSS_{\alpha,1}$. We noted, though, that this direction is not needed for the proof of Theorem \ref{main_theorem}. 

\medskip

Part (2) of Theorem \ref{thm_jump_inversion} follows from Proposition \ref{prop_jump_inversion_part_two}. For any countable graphs $G$ and $H$, for any nonstandard $\alpha,\beta \in \delta^*\setminus \omck$, we have $G^{-\alpha} \cong H^{-\beta}$; this is because $\SSS_{\alpha,0}\cong \SSS_{\alpha,1}\cong \SSS_{\beta,0}\cong\SSS_{\beta,1}$. Note that this structure, $G^{-\infty}$, has a computable copy. 

\medskip

We turn to the last part of the theorem. Given a total Turing functional $\Phi$, we need to construct a Turing functional $\Psi$ such that for all $X\in 2^\w$, $\alpha<\delta^*$, and all graphs $G$, if $\Phi(X,\emptyset_{(2\alpha)})\cong G$ then $\Psi(X,\alpha)\cong G^{-\alpha}$. We use Proposition~\ref{prop_jump_inversion_part_three}. Given $X$ and $\alpha$, let $\seq{v_n}_{n<\w}$ be an $X\oplus \emptyset_{(2\alpha)}$-computable enumeration of the vertices of $\Phi(X,\emptyset_{(2\alpha)})$. For any $n,m<\w$ we can ask $X\oplus \emptyset_{(2\alpha)}$ whether the edge $(v_n,v_m)$ belongs to the graph $\Phi(X,\emptyset_{(2\alpha)})$; this is uniform in $X$, $\alpha$, $n$ and $m$, and by assumption we always get an answer. Proposition \ref{prop_jump_inversion_part_three} gives us a functional such that for all $n,m<\w$, $X\in 2^\w$ and nonzero $\alpha<\delta^*$, outputs a copy of $\SSS_{\alpha,1}$ if the edge $(v_n,v_m)$ is in the graph $G=\Phi(X,\emptyset_{(2\alpha)})$, and otherwise gives a copy of $\SSS_{\alpha,0}$. From this functional we can easily build a copy of $G^{-\alpha}$.

\section{A Linear ordering}

We prove the following extension of Theorem \ref{main_theorem}:

\begin{theorem}\label{thm_LO}
	There is a countable linear ordering $\LL$ whose degree spectrum consists of the non-hyperarithmetic degrees. 
\end{theorem}

Using the notation of the previous section, we make use of the following relativisation of a uniform version of a Theorem of Ash's (this is Theorem 18.15 of \cite{AK00}). Below, for $X\in 2^\w$ and an $X$-computable ordinal $\delta$, we again identify $\delta$ with some $X$-notation for $\delta$, from which we also derive an $X$-computable well-ordering of $\w$ of order-type $\delta$. 

\begin{theorem} \label{thm_Ash_LO}
	Let $X\in 2^\w$. For any linear ordering $\LL$ and any $\alpha<\w_1^X$, $\LL$ has an $X_{(2\alpha)}$-computable copy if and only if $\w^\alpha\cdot \LL$ has an $X$-computable copy. 
	
	This is uniform in $\alpha$: let $\delta<\w_1^X$ be an $X$-computable ordinal. Suppose that $\Phi$ is a Turing functional such that for all $\alpha<\delta$, $\LL_\alpha = \Phi(X_{(2\alpha)},\alpha)$ is a linear ordering. Then there is a Turing functional $\Psi$ such that for all $\alpha<\delta$, $\Psi(X,\alpha)$ is a linear ordering isomorphic to $\w^{\alpha}\cdot \LL_\alpha$. 
\end{theorem}

We also make use of a result of Frolov, Harizanov, Kalimullin, Kudinov and Miller from \cite{FHKKM}. They combined coding families of sets in a shuffle sum of linear orderings with a result of Wehner's (similar to the one we used above) to show that there is a linear ordering whose degree spectrum is the collection of nonlow$_2$ degrees.

\begin{theorem}\label{thm_FHKKM}
	For every set $Y$ there is a linear ordering $\LL_Y$ such that for all $X\ge_\Tur Y$, $X$ computes a copy of $\LL_Y$ if and only if $X''>_\Tur Y''$. 
	
	This is uniform in $Y$: there is a Turing functional $\Theta$ such that for all $Y$ and all $X$ such that $X''>_\Tur Y''$, $\Theta(X,Y)\cong \LL_Y$. 
\end{theorem}


Because adding an extremal point does not change the degree spectrum of a linear ordering, we may assume that every linear ordering $\LL_Y$ has a least element.

\

The ``ill-founded limit'' of the linear orderings $\w^\alpha\cdot \LL$ is $\omck\cdot \Rat$, the linear ordering which contains densely many copies of $\omck$. To see this, we need to use a generalisation of ordinal exponentiation $\w^\alpha$ to ill-founded linear orders. For any linear order $\LL$, the linear ordering $\w^{\LL}$ is the collection of all functions $f\colon \LL\to \w$ which take the value 0 on all but finitely many inputs, ordered by a ``lexicographic'' ordering: $f<_{\w^{\LL}}g$ if for the $\LL$-least $x$ such that $f(x)\ne g(x)$ we have $f(x)<g(x)$ (in $\w$, of course). That this is indeed a generalisation of taking ordinal powers of $\w$ can be seen by considering an inductive definition of a linear ordering (of order-type) $\w^\beta$ for ordinals $\beta$, using a directed system of linear orderings. At the successor step we let $\w^{\beta+1}= \w^\beta\cdot \w$, with the embedding taking $\w^\beta$ to the leftmost copy $\w^\beta\times \{0\}$ of $\w^\beta$ in $\w^{\beta}\cdot \w$; at limit stages we take direct limits. The direct construction of the linear ordering $\w^{\LL}$ shows that a power rule holds: for linear orderings $\LL$ and $\KK$, $\w^{\LL+\KK}\cong \w^\LL\cdot \w^{\KK}$. 

\begin{proposition}\label{prop_ill_founded_LO}
	Let $\delta$ be an ill-founded linear ordering whose maximal well-founded initial segment has order-type $\omck$. Let $\LL$ be an countable linear ordering. Then $\w^\delta\cdot \LL$ is isomorphic to $\omck\cdot \Rat$. 
\end{proposition}

\begin{proof}
	Let $C$ be the maximal well-founded initial segment of $\delta$; so $\delta-C$ has no least element. This means that $\w^{\delta-C}$ is dense (with no endpoints). This, in turn, means that $\w^{\delta-C}\cdot \LL$ is also dense, and so is isomorphic to the rationals. Since $\omck = \w^{\omck}$, 
	\[ \w^\delta\cdot \LL \,\cong\, \w^{C}\cdot \w^{\delta-C}\cdot \LL \,\cong\, \w^{\omck}\cdot \Rat \,\cong\, \omck\cdot\Rat. \qedhere\]
\end{proof}

Armed with Theorems \ref{thm_Ash_LO} and \ref{thm_FHKKM}, we give a proof of Theorem \ref{thm_LO}. 
The linear ordering in question is the sum
\[ \KK =  \sum_{\alpha<\omck}  \w^\alpha\cdot \LL_{\emptyset_{(2\alpha)}}  \,\, + \,\,  \omck\cdot\Rat.\]

Because $\LL_{\emptyset_{(2\alpha)}}$ has a least element, so does $\w^\alpha\cdot \LL_{\emptyset_{(2\alpha)}}$. This means that if a set $X$ computes a copy of $\KK$, then it computes a copy of $\w^\alpha\cdot \LL_{\emptyset_{(2\alpha)}}$ for all $\alpha<\omck$. In turn, this means that for all $\alpha<\omck$, $X_{(2\alpha)}$ computes a copy of $\LL_{\emptyset_{(2\alpha)}}$, which means that $X_{(2\alpha+2)}>_\Tur \emptyset_{(2\alpha+2)}$. As we have seen in the previous section, it follows that $X$ cannot be hyperarithmetic: if $X\le_\Tur \emptyset_{(\beta)}$, then for all $\alpha>\beta\cdot \w$ we get $X_{(\alpha)} \equiv_\Tur \emptyset_{(\alpha)}$. 

\

It remains to show, then, that any nonhyperarithmetic set can compute a copy of $\KK$. This is done slightly differently to the way we argued in the previous section. Here we will see that it is important that the sum $\KK$ is unlabelled. In the previous section it was not important that for nonstandard $\alpha$, the graph $G^{-\alpha}$ did not depend on $\alpha$; it was just important that it did not depend on $G$. Here it is important that for nonstandard $\alpha$, the linear ordering $\w^{\alpha}\cdot \LL$ depends on neither $\LL$ nor $\alpha$. 

The reason for this is that we cannot argue, as we did in the previous section, for all nonhyperarithmetic sets at once. We do not produce a single functional which outputs a copy of $\KK$ given any nonhyperarithmetic set. This is because we use a stronger form of overspill to stretch Ash's theorem beyond $\omck$. We cannot obtain it for all sets $X$ at once, as that is a $\Pi^1_1$ statement. We stretch Ash's theorem for each $X$ separately, and this means that we have to treat two distinct cases: $\w_1^X = \omck$ and $\w_1^X>\omck$.

\medskip

Let $X$ be a nonhyperarithmetic set. If $\w_1^X = \omck$, then an application of the overspill principle (equivalently, working in an ill-founded $\w$-model of set theory $H$ which contains $X$ and omit $\omck$) yields a nonstandard $X$-computable ordinal $\delta^*$ whose maximal well-founded initial segment has order-type $\omck$, which supports a jump hierarchy and for which Theorem \ref{thm_Ash_LO} holds: 
\begin{itemize}
	\item For all $Y\le_\Tur X$ there are sets $\seq{Y_{(\alpha)}}$ for $\alpha<\delta^*$ which obey the recursive definition of a transfinite iteration of the Turing jump; and
	\item If $\Phi$ is a Turing functional such that for all $\alpha<\delta^*$, $\Phi(X_{(2\alpha)}, \alpha)$ is a linear ordering $\LL_\alpha$, then there is a Turing functional $\Psi$ such that for all $\alpha<\delta^*$, $\Psi(X,\alpha)$ is isomorphic to $\w^{\alpha}\cdot \LL_\alpha$. 
\end{itemize}

In our application, we use $\LL_\alpha = \Theta(X_{(2\alpha)}, \emptyset_{(2\alpha)})$, where $\Theta$ is given by Theorem \ref{thm_FHKKM}. For standard $\alpha<\omck$, because $X$ is not hyperarithmetic, $X_{(2\alpha+2)}>_\Tur \emptyset_{(2\alpha+2)}$, and so for standard $\alpha$ we get $\LL_\alpha\cong \LL_{\emptyset_{(2\alpha)}}$. Because $\emptyset_{(2\alpha)}$ is computable from $X_{(2\alpha)}$ (even for nonstandard $\alpha$), uniformly in $\alpha$, we see that there is indeed a Turing functional $\Phi$ such that for all $\alpha<\delta^*$, $\Phi(X_{(2\alpha)},\alpha) = \LL_\alpha$. Using the functional $\Phi$ obtained from $\Psi$, we get, uniformly in $\alpha<\delta$ with oracle $X$, a copy $\Psi(X,\alpha)$ of $\w^\alpha\cdot \LL_\alpha$. In this way,
\[ \KK(X) = \sum_{\alpha<\delta^*}  \Psi(X,\alpha)  \]
is computable from $X$. It is easy to show that $\KK(X)$ is isomorphic to $\KK$: for standard $\alpha<\omck$ we have $\Psi(X,\alpha)\cong \w^\alpha\cdot \LL_{\emptyset_{(2\alpha)}}$, for nonstandard $\alpha \in \delta^*\setminus \omck$ we have $\Psi(X,\alpha)\cong \omck\cdot \Rat$ (Proposition \ref{prop_ill_founded_LO}), and any sum of copies of $\omck\cdot\Rat$ is again isomorphic to $\omck\cdot\Rat$.

\medskip

Now suppose that $\w_1^X>\omck$. We can then fix an $X$-computable copy of $\omck$, which we naturally also call $\omck$. The argument is now simpler. Applying Theorem \ref{thm_Ash_LO} to $\omck$, the argument above, using the fact that $X$ is not hyperarithmetic, gives us, uniformly in $\alpha<\omck$, an $X$-computable copy of $\w^\alpha\cdot \LL_{\emptyset_{(2\alpha)}}$, and so $X$ computes a copy of $\sum_{\alpha<\omck}\w^\alpha\cdot\LL_{\emptyset_{(2\alpha)}}$. Because $X$ computes a copy of $\omck$ and of course computes $\Rat$, it also computes a copy of $\omck\cdot\Rat$. Putting these together, we see that $X$ computes a copy of $\KK$. This completes the proof of Theorem \ref{thm_LO}.

\

\

\section{Distinguishing category and measure} \label{sec_separating}

\begin{theorem}\label{thm_null_co_meagre}
	There is a structure whose degree spectrum is null and co-meagre. 
\end{theorem}

\begin{proof}
	In \cite{ShinodaSlaman}, Shinoda and Slaman show that if $\mathcal C$ is a $\Pi^0_2(\cero')$ co-meagre class defined as the intersection of uniformly $\Sigma^0_2$ dense open classes, then there is a co-meagre class $\mathcal D\subseteq \mathcal C$ such that the upward closure of $\mathcal D$ in the Turing degrees is a degree spectrum. In \cite{Downey_Jockusch_Stob:_array_recursive_2}, Downey, Jockusch and Stob show that there is a $\Pi^0_2(\cero')$ co-meagre class as above: the collection of \emph{pb-generic} sets, whose upward closure in the Turing degrees is the collection of array noncomputable degrees. Hence there is a co-meagre degree spectrum contained in the array non-computable degrees. The theorem follows from the fact that the collection of array noncomputable degrees is null.
\end{proof}

\

In the rest of this section, we prove the following:

\begin{theorem}\label{thm_meagre_co_null}
	There is a structure whose degree spectrum is meagre and co-null. 
\end{theorem}

Theorem \ref{thm_meagre_co_null} follows from applying Lemma \ref{lem_flower_grpahs_and_enumerations} to the family $\SSS$ given by the following theorem:

\begin{theorem}\label{thm_meagre_co_null_enumeration}
	There is a countable family $\SSS$ of subsets of $\Nat$ such that every 1-random set can enumerate $X$, but no 2-generic set can enumerate $\SSS$. 
\end{theorem}

To prove Theorem \ref{thm_meagre_co_null_enumeration}, let $\seq{\epsilon_{e,\s}}_{e<\w,\s\in 2^{<\w}}$ be a computable array of positive rational numbers whose sum $\sum_{e,\s} \epsilon_{e,\s}$ is smaller than 1. For every $e<\w$ and $\s\in 2^{<\w}$ we define a countable family $\SSS_{e,\s}$ and uniformly, a c.e.\ operator $\Lambda_{e,\s}$ and a $\Pi^0_1$ class $\PP_{e,\s}$ whose measure is at least $1-\epsilon_{e,\s}$ such that for all $X\in \PP_{e,\s}$, $\Lambda_{e,\s}(X)$ is an enumeration of $\SSS_{e,\s}$. We then let 
\[ \SSS = \bigoplus_{e,\s} \SSS_{e,\s} = \big\{ \{(e,\s)\}\oplus A  \,:\, A\in \SSS_{e,\s} \big\}.\]
The uniformity of the array $\Lambda_{e,\s}$ shows that every $X$ in the $\Pi^0_1$ class $\PP = \bigcap_{e,\s}\PP_{e,\s}$ can enumerate $\SSS$. The measure of $\PP$ is at least $1-\sum_{e,\s}\epsilon_{e,\s}$, and so $\PP$ is non-null. By Ku\v{c}era's \cite{Kucera:85}, every 1-random set is Turing equivalent to some element of~$\PP$, and so every 1-random set can enumerate $\SSS$.

Fix some $e$. We use the families $\SSS_{e,\s}$ to ensure that if $G$ is 2-generic, then $\Phi_e(G)$ is not an enumeration of $\SSS$; here $\Phi_e$ is the $e\tth$ c.e.\ operator. From $\Phi_e$ and $\s$ we can find a c.e.\ operator $\Psi_{e,\s}$ such that for any $X$, if $\Phi_e(X)$ is an enumeration of $\SSS$, then $\Psi_{e,\s}(X)$ is an element of $\SSS_{e,\s}$. We would like it to be the case that if $G$ is a generic set which extends $\s$, then $\Psi_{e,\s}(G)$ is not an element of $\SSS_{e,\s}$. We will not always be able to achieve this aim. We will be able to ensure that there is some $\tau$ extending $\s$ such that for a sufficiently generic set $G$ (1-generic will be enough), if $G$ extends $\tau$ then $\Psi_{e,\s}(G)\notin \SSS_{e,\s}$. As we do this for each $\s$, the collection of $\tau$ which force that $\Psi_{e,\s}(G)\notin \SSS_{e,\s}$ for \emph{some} $\s$ is dense, and so if $G$ is sufficiently generic (2-generic will suffice), $\Phi_e(G)$ cannot be an enumeration of $\SSS$. Hence the collection of oracles $X$ that can enumerate $\SSS$ is meagre.

Fixing $e$ and $\s$, let $\seq{\tau_k}_{k<\w}$ be an enumeration of all extensions of $\s$. Let $\delta_k = \delta_k(e,\s)$ be positive binary rational numbers such that $\sum_{k<\w} \delta_k < \epsilon_{e,\s}$; let $n_k = 1/\delta_k$, so $n_k$ is a power of 2. Partition $\w$ computably into sets $I_k$ such that $|I_k| = {n_k}$. 

At step $k$, working on $\tau_k$, we first tempt a generic set $G$ extending $\tau_k$ to enumerate an element of $I_k$ into $\Psi_{e,\s}(G)$. We will need to ensure that if $G$ does not oblige, that is, if we find no extension $\rho$ of $\tau_k$ for which there is some $x_k\in I_k\cap \Psi_{e,\s}(\rho)$, then we will have already ensured that $\Psi_{e,\s}(G)\notin \SSS_{e,\s}$. We do this by enumerating elements of $I_k$ into each set in $\SSS_{e,\s}$. If we can find such an extension for every $\tau_k$, then a 1-generic set $G$ will contain $x_k$ for infinitely many $k$. We then ensure that no set in $\SSS_{e,\s}$ contains infinitely many of the numbers $x_k$. To ensure this, we will need to ``throw away'' sets $X$ which enumerate too many $x_k$'s into sets in $\Lambda_{e,\s}(X)$; we need to make sure that we do not get rid of too many (in the sense of measure) such sets $X$. 

Note that while we are waiting for $x_k$ to be defined (which may not happen), that is, while we are waiting for an extension $\rho$ of $\tau_k$ which enumerates an element of $I_k$ into $\Psi_{e,\s}(\rho)$, we need to enumerate various elements of $I_k$, among them $x_k$, into sets in $\Lambda_{e,\s}(X)$. In advance, we cannot ensure that any particular element of $I_k$ is enumerated by only few oracles, because if $x_k$ is never defined, we still need to ensure that $\Lambda_{e,\s}(X)$ enumerates the same family $\SSS_{e,\s}$ for most oracles $X$. Thus $x_k$ will be an elements of some sets in $\SSS_{e,\s}$. But by passing to new sets (in a sense using some kind of priority between the ``old'' and the ``new'' parts of $\SSS_{e,\s}$) we can ensure that each element of $S_{e,\s}$ contains only finitely many $x_k$. 

So for $k<\w$, let $x_k$ be the first number discovered such that there is some $\rho\supseteq \tau_k$ with $x_k\in I_k\cap \Psi_{e,\s}(\rho)$. If there is some $k$ for which $x_k$ is not defined (for all $\rho\supseteq \tau_k$, $I_k\cap \Psi_{e,\s}(\rho)$ is empty), let $k^* = k^*(e,\s)$ be the least such $k$. Otherwise, let $k^*=\w$. We can now define $\SSS_{e,\s}$:
\begin{itemize}
	\item If $k^* = \w$, then $A\in \SSS_{e,\s}$ if for some $k<\w$, $I_k\cap A$ is a singleton, and for all $j> k$, $A\cap I_j = I_j\setminus \{x_j\}$. 
	\item If $k^*< \w$, then $A\in \SSS_{e,\s}$ if $A\cap I_{k^*}$ is a singleton, and for all $j>k^*$, $A\cap I_j$ is empty. 
\end{itemize}

The family $\SSS_{e,\s}$ is indeed countable. 

\begin{claim} \label{clm_kstar}\
	\begin{enumerate}
		\item If $k^*=\w$ and $G\supset \s$ is 1-generic, then $\Psi_{e,\s}(G)\notin \SSS_{e,\s}$. 
		\item If $k^*<\w$ and $G\supset \tau_{k^*}$, then $\Psi_{e,\s}(G)\notin \SSS_{e,\s}$. 
	\end{enumerate}
\end{claim}

\begin{proof}
	If $k^*<\w$, then for all $G\supset \tau_{k^*}$, $I_{k^*}\cap \Psi_{e,\s}(G)$ is empty. On the other hand, for all $A\in \SSS_{e,\s}$, $I_{k^*}\cap A$ is nonempty. 
	
	Suppose that $k^*=\w$, and that $G\supset \s$ is 1-generic. Then for infinitely many $k$, $x_k\in \Psi_{e,\s}(G)$, whereas for all $A\in \SSS_{e,\s}$, $x_k\in A$ for only finitely many $k$. 
\end{proof}

\begin{claim}
	If $G$ is 2-generic, then $G$ cannot enumerate $\SSS$.
\end{claim}

\begin{proof}
	Let $G$ be 2-generic, and let $e<\w$. We show that there is some $\s$ such that $\Psi_{e,\s}(G)\notin \SSS_{e,\s}$, which shows that $\Phi_e(G)$ is not an enumeration of $\SSS$. 
	
	If there is some $\s\subset G$ such that $k^*(e,\s)=\w$, then by Claim \ref{clm_kstar}(1), $\Psi_{e,\s}(G)\notin \SSS_{e,\s}$. Otherwise, consider the set $D$ of strings $\tau_{k^*(e,\s)}$ as $\s$ ranges over the strings such that $k^*(e,\s)<\w$. Then $D$ is a $\Pi^0_1$ collection of strings which is dense around $G$, and so $G$ has some initial segment in $D$. By Claim \ref{clm_kstar}(2), $\Psi_{e,\s}(G)\notin \SSS_{e,\s}$ for some initial segment $\s$ of $G$. 	
\end{proof}

\

We turn to the definition of $\Lambda_{e,\s}$. For any $X\in 2^\w$, we think of $\Lambda_{e,\s}(X)$ as enumerating sets $A_{k,x,F}(X)$, indexed by $k\le k^*$ (of course we mean $k<\w$ if $k^* = \w$), $x\in I_k$ and $F\subseteq \bigcup_{j< k}I_j$. For brevity, let $Q_k = I_k\times \PP(\bigcup_{j<k}I_j)$ be the set of pairs $(x,F)$ with $x\in I_k$ and $F\subseteq \bigcup_{j<k}I_j$. Let $\seq{s_k}_{k<k^*}$ be a computable increasing sequence of stages such that at stage $s_k$ we observe $x_k$; let $s_{-1}=0$. Then in practice, at stage $s_{k-1}$, we associate each triple $(k,x,F)$ for $(x,F)\in Q_k$ with a fresh index $n$, and will let $\Lambda_{e,\s}(X)^{[n]} = A_{k,x,F}(X)$. At stage $s$ which is not equal to $s_k$ for any $k$, if $\Lambda_{e,\s}(X)^{[s]}$ has not yet been associated with any set $A_{k,x,F}(X)$, then we declare that $\Lambda_{e,\s}(X)^{[s]} = \Lambda_{e,\s}(X)^{[s-1]}$. In this way we ensure that $\Lambda_{e,\s}(X)$ is an enumeration of the family $\GG_{e,\s}(X) = \left\{ A_{k,x,F}(X)  \,:\, k\le k^* \andd (x,F)\in Q_k \right\}$. 

Let $k\ge 0$. At stage $s_{k-1}$ we do the following. 
\begin{itemize}
	\item For all $X\in 2^\w$ and all $(x,F)\in Q_{k}$, enumerate $F\cup \{x\}$ into $A_{k,x,F}(X)$. 
	\item If $k\ge 1$, partition Cantor space $2^\w$ into $n_{k}$ many clopen sets $\CC_{x,k}$ for $x\in I_{k}$, each of measure $\delta_{k}$. Let $x\in I_{k}$; for all $j\le k$, for all $(y,F)\in Q_j$, and all $X\in \CC_{x,k}$, enumerate $x$ into $A_{j,y,F}(X)$. 
	\item If $k\ge 2$, for all $j<k-1$ and all $(y,F)\in Q_j$, for all $X\in 2^\w\setminus \CC_{x_{k-1},k-1}$, enumerate all of $I_{k-1}\setminus\{x_{k-1}\}$ into $A_{j,y,F}(X)$. 
\end{itemize}

This defines the family $\GG_{e,\s}(X)$ of sets $A_{k,x,F}(X)$ for all $X$, and so the functional $\Lambda_{e,\s}$. For positive $k<k^*$, let $\PP_k = \PP_k(e,\s) = 2^\w\setminus \CC_{x_k,k}$; let $\PP_{e,\s} = \bigcap_{k<k^*}\PP_k$. The class $\PP_{e,\s}$ is a $\Pi^0_1$ class (uniformly in $e$ and $\s$), and the measure of $\PP_{e,\s}$ is at least $1-\sum_{k<\w}\delta_k \ge 1 - \epsilon_{e,\s}$. What remains is to show that for all $X\in \PP_{e,\s}$, $\Lambda_{e,\s}(X)$ is an enumeration of $\SSS_{e,\s}$, that is, that for each $X\in \PP_{e,\s}$ we have $\SSS_{e,\s} = \GG_{e,\s}(X)$. 

\begin{claim} \label{clm_meagre_conull_main}
	Let $X\in \PP_{e,\s}$. Let $k\le k^*$ and $(x,F)\in Q_k$.
	\begin{enumerate}
		\item If $k^* = \w$, then $A_{k,x,F}(X)  = F \cup \{x\} \cup\bigcup_{j> k} (I_j\setminus \{x_j\})$.
		\item If $k^*< \w$, then $A_{k,x,F}(X) = B\cup \{y\}$ where $(y,B)\in Q_{k^*}$. If $k=k^*$ then $A_{k^*,x,F} = F\cup\{x\}$.
	\end{enumerate}
\end{claim}

\begin{proof}
	Suppose that $k^*=\w$. At stage $s_{k-1}$, we enumerate $F\cup\{x\}$ into $A_{k,x,F}(X)$; at later stages we do not enumerate any other element of $\bigcup_{j\le k}I_j$ into $A_{k,x,F}(X)$. Let $j>k$. At stage $s_{j-1}$ we enumerate $y$ into $A_{k,x,F}(X)$, where $X\in \CC_{y,j}$. Since $X\notin \CC_{x_j,j}$, we have $y\ne x_j$. At stage $s_j$ we enumerate the rest of $I_j\setminus \{x_j\}$ into $A_{k,x,F}(X)$; at no later stage is $x_j$ enumerated into $A_{k,x,F}(X)$.
	
	Now suppose that $k^*<\w$. At stages $s_{j-1}$ for $j<k^*$, only numbers in $\bigcup_{i<k^*}I_i$ are enumerated into $A_{k,x,F}$. If $k<k^*$, then at stage $s_{k^*-1}$ we enumerate $y$ into $A_{k,x,F}(X)$, where $X\in \CC_{y,k^*}$; no other element of $I_{k^*}$ is enumerated into $A_{k,x,F}(X)$. If $k=k^*$ then at stage $s_{k^*-1}$ we enumerate $F\cup \{x\}$ into $A_{k,x,F}(X)$. In either case, as $s_{k^*-1}$ is the last stage at which we enumerate any numbers into $A_{k,x,F}(X)$, we see that $I_{k^*}\cap A_{k,x,F}(X)$ is a singleton; and that for all $j>k^*$, no element of $I_j$ is ever enumerated into $A_{k,x,F}(X)$.
\end{proof}

A short examination of the definition of $\GG_{e,\s}$ now shows that for all $X\in \PP_{e,\s}$, $\GG_{e,\s}(X) = \SSS_{e,\s}$. This completes the proof of Theorem \ref{thm_meagre_co_null_enumeration}.

\bibliography{bftypes}

\newcommand{\etalchar}[1]{$^{#1}$}
\begin{thebibliography}{GHK{\etalchar{+}}05}

\bibitem[AK90]{AK90}
C.~J. Ash and J.~F. Knight.
\newblock Pairs of recursive structures.
\newblock {\em Ann. Pure Appl. Logic}, 46(3):211--234, 1990.

\bibitem[AK00]{AK00}
C.J. Ash and J.~Knight.
\newblock {\em Computable Structures and the Hyperarithmetical Hierarchy}.
\newblock Elsevier Science, 2000.

\bibitem[CK10]{CsimaKalimullin}
Barbara~F. Csima and Iskander~S. Kalimullin.
\newblock Degree spectra and immunity properties.
\newblock {\em MLQ Math. Log. Q.}, 56(1):67--77, 2010.

\bibitem[DJS96]{Downey_Jockusch_Stob:_array_recursive_2}
Rod Downey, Carl~G. Jockusch, and Michael Stob.
\newblock Array nonrecursive degrees and genericity.
\newblock In {\em Computability, enumerability, unsolvability}, volume 224 of
  {\em London Math. Soc. Lecture Note Ser.}, pages 93--104. Cambridge Univ.
  Press, Cambridge, 1996.

\bibitem[GHK{\etalchar{+}}05]{GHKMMS}
Sergey Goncharov, Valentina Harizanov, Julia Knight, Charles McCoy, Russell
  Miller, and Reed Solomon.
\newblock Enumerations in computable structure theory.
\newblock {\em Ann. Pure Appl. Logic}, 136(3):219--246, 2005.

\bibitem[GMS11]{GMS}
Noam Greenberg, Antonio Montalb{\'a}n, and Theodore~A. Slaman.
\newblock The {S}laman-{W}ehner theorem in higher recursion theory.
\newblock {\em Proc. Amer. Math. Soc.}, 139(5):1865--1869, 2011.

\bibitem[Har68]{Har68}
J.~Harrison.
\newblock Recursive pseudo-well-orderings.
\newblock {\em Transactions of the American Mathematical Society},
  131:526--543, 1968.

\bibitem[Har74]{Harrington_prime}
Leo Harrington.
\newblock Recursively presentable prime models.
\newblock {\em J. Symbolic Logic}, 39:305--309, 1974.

\bibitem[Hir06]{Hir06}
Denis~R. Hirschfeldt.
\newblock Computable trees, prime models, and relative decidability.
\newblock {\em Proc. Amer. Math. Soc.}, 134(5):1495--1498 (electronic), 2006.

\bibitem[HW02]{HW02}
Denis~R. Hirschfeldt and Walker~M. White.
\newblock Realizing levels of the hyperarithmetic hierarchy as degree spectra
  of relations on computable structures.
\newblock {\em Notre Dame J. Formal Logic}, 43(1):51--64 (2003), 2002.

\bibitem[Kal07]{Kalimullin-low}
I.~Sh. Kalimullin.
\newblock Spectra of degrees of some algebraic structures.
\newblock {\em Algebra Logika}, 46(6):729--744, 793, 2007.

\bibitem[Kal08a]{Kalimullin-ce}
I.~Sh. Kalimullin.
\newblock Almost computably enumerable families of sets.
\newblock {\em Mat. Sb.}, 199(10):33--40, 2008.

\bibitem[Kal08b]{Kalimullin-restrict}
I.~Sh. Kalimullin.
\newblock Restrictions on the spectra of degrees of algebraic structures.
\newblock {\em Sibirsk. Mat. Zh.}, 49(6):1296--1309, 2008.

\bibitem[Khi74]{Khisamiev}
N.~G. Khisamiev.
\newblock Strongly constructive models of a decidable theory.
\newblock {\em Izv. Akad. Nauk Kazah. SSR Ser. Fiz.-Mat.}, (1):83--84, 94,
  1974.

\bibitem[Kho86]{Bakh_dasies}
Bakhadyr~M. Khoussainov.
\newblock Strongly effective unars and nonautoequivalent constructivizations.
\newblock In {\em Some problems in differential equations and discrete
  mathematics ({R}ussian)}, pages 33--44. Novosibirsk. Gos. Univ., Novosibirsk,
  1986.

\bibitem[Ku{\v{c}}85]{Kucera:85}
Anton{\'{\i}}n Ku{\v{c}}era.
\newblock Measure, {$\Pi^0_1$}-classes and complete extensions of {${\rm PA}$}.
\newblock In {\em Recursion theory week ({O}berwolfach, 1984)}, volume 1141 of
  {\em Lecture Notes in Math.}, pages 245--259. Springer, Berlin, 1985.

\bibitem[Ner57]{Nerode:57}
Anil Nerode.
\newblock General topology and partial recursive functionals.
\newblock In {\em Summaries of talks presented at the summer institute for
  symbolic logic, 1957, Cornell University}, pages 247--251, 1957.

\bibitem[Sac90]{Sacks}
Gerald~E. Sacks.
\newblock {\em Higher recursion theory}.
\newblock Perspectives in Mathematical Logic. Springer-Verlag, Berlin, 1990.

\bibitem[Sla98]{Sla98}
Theodore~A. Slaman.
\newblock Relative to any nonrecursive set.
\newblock {\em Proc. Amer. Math. Soc.}, 126(7):2117--2122, 1998.

\bibitem[SS00]{ShinodaSlaman}
Juichi Shinoda and Theodore~A. Slaman.
\newblock Recursive in a generic real.
\newblock {\em J. Symbolic Logic}, 65(1):164--172, 2000.

\bibitem[Weh98]{Weh98}
Stephan Wehner.
\newblock Enumerations, countable structures and {T}uring degrees.
\newblock {\em Proc. Amer. Math. Soc.}, 126(7):2131--2139, 1998.

\end{thebibliography}
\bibliographystyle{alpha}

\end{document}